\documentclass[11pt]{article}

\newcommand{\insieme}[1]{\left\{ #1 \right\}}
\usepackage{fullpage}

\usepackage{enumerate}
\usepackage{amsmath,amsfonts,amssymb,amsthm, bbold}
\usepackage{svg}
\usepackage{graphics,esint}
\usepackage{epsfig}
\usepackage{xcolor}
\usepackage{xspace}

\definecolor{pingreen}{rgb}{0,39,14}

\usepackage{xfrac}
\usepackage{comment}
\usepackage{hyperref}
\usepackage{cleveref}

\crefname{section}{§}{§§}
\Crefname{section}{§}{§§}

\def\lt{\left}
\def\vol{\mathrm{vol}}

\def\rt{\right}

\def\Acal{\mathcal{A}}

\def\rhs{r.h.s.\xspace}

\def\st{s.t.\xspace}

\makeatletter
\newtheorem*{rep@theorem}{\rep@title}
\newcommand{\newreptheorem}[2]{%
\newenvironment{rep#1}[1]{%
 \def\rep@title{#2~\ref{##1}}%
 \begin{rep@theorem}}%
 {\end{rep@theorem}}}
\makeatother

\newtheorem{theorem}{Theorem}[section]

\newtheorem{lemma}[theorem]{Lemma}
\newtheorem{definition}[theorem]{Definition}

\newtheorem{remark}[theorem]{Remark}

\newtheorem{proposition}[theorem]{Proposition}

\newtheorem{corollary}[theorem]{Corollary}

\newreptheorem{theorem}{Theorem}

\newcommand{\R}{\mathbb{R}}
\newcommand{\Z}{\mathbb{Z}}

\newcommand{\Fcal}{\mathcal{F}}

\def\N{\mathbb N}
\def\eps{\varepsilon}

\def\per{\mathrm{Per}}

\def\eps{\varepsilon}

\def\d {\,\mathrm {d}}
\def\dx{\,\mathrm {d}x}
\def\dz{\,\mathrm {d}z}
\def\ds{\,\mathrm {d}s}
\def\du{\,\mathrm {d}u}
\def\dv{\,\mathrm {d}v}
\def\dt{\,\mathrm {d}t}
\def\dy{\,\mathrm {d}y}

\def\FtL{\mathcal F_{\tau,L}}

\numberwithin{equation}{section}

\usepackage{authblk}

\author[1]{Sara Daneri\thanks{sara.daneri@gssi.it}}
\author[2]{Eris Runa\thanks{eris.runa@gmail.com}}
\affil[1]{Gran Sasso Science Institute, L'Aquila }
\affil[2]{Deutsche Bank, London}

  \title{Exact periodic stripes for a local/nonlocal minimization problem with volume constraint}

\usepackage[
doi=false,
url=false,
eprint=false,
style=alphabetic,
maxcitenames=6,
maxbibnames=6,
maxalphanames=6,
]{biblatex}

\bibliography{allbib}

\date{}

\begin{document}

\maketitle

\begin{abstract}
	We consider a class of generalized antiferromagnetic local/nonlocal interaction functionals in general dimension, where a short range attractive term of perimeter type competes with a long range repulsive term characterized by a reflection positive power law kernel. Breaking of symmetry with respect to coordinate permutations and pattern formation for functionals in this class have been shown in~\cite{gr,dr_arma} and previously by~\cite{gs_cmp} in the discrete setting, for a smaller range of exponents.  Global minimizers of such functionals have been proved in~\cite{dr_arma} to be given by periodic stripes of volume density $1/2$ in any cube having optimal period size, also in the large volume limit.  In this paper  we study the minimization problem with arbitrarily prescribed volume constraint $\alpha\in(0,1)$. We show that, in the large volume limit, minimizers are periodic stripes of volume density $\alpha$, namely stripes whose one-dimensional slices in the direction orthogonal to their boundary are simple periodic with volume density $\alpha$ in each period. Results of this type in the one-dimensional setting, where no symmetry breaking occurs, have been previously obtained in \cite{muller1993singular, alberti2001new,ren2003energy,chen2005periodicity,giuliani2009modulated}.
\end{abstract}

\section{Introduction}
\label{sec:introduction}

For  $d\geq1$, $L>0$, $\tau>0$, $p\geq d+2$, $\beta=p-d-1$, $E\subset\R^d$ $[0,L)^d$-periodic set, $Q_L=[0,L)^d$ consider the following functional
\begin{equation}
   \label{eq:ftauel}
\mathcal F_{\tau,L}(E)=\frac{1}{L^d}\Big(-\per_1(E,Q_L)+\int_{\R^d} K_\tau(\zeta) \Big[\int_{\partial E \cap Q_L} \sum_{i=1}^d|\nu^E_i(x)| |\zeta_i|\d\mathcal H^{d-1}(x)-\int_{Q_L}|\chi_E(x)-\chi_E(x+\zeta)|\dx\Big]\d\zeta\Big),
\end{equation}
where  $\per_1(E,Q_L)=\int_{\partial E\cap Q_L}\|\nu^E(x)\|_1\d\mathcal H^{d-1}(x)$ is the $1$-perimeter of the set $E$ in the cube $Q_L$,  defined through the $1$-norm $\|z\|_1=\sum_{i=1}^d|z_i|$, and  $K_\tau(\zeta)=\tau^{-p/\beta}K_1(\zeta\tau^{-1/\beta})$, where $K_1(\zeta)=\frac{1}{(\|\zeta\|_1+1)^p}$.

The functional~\eqref{eq:ftauel} is obtained by rescaling the local/nonlocal interaction functional
\begin{equation}
	\label{eq:ef}
	\bar\Fcal_{J,L}(E)=J\per_1(E,[0,L)^d)-\int_{\R^d}\int_{Q_L}|\chi_E(x)-\chi_E(x+\zeta)|K_1(\zeta)\dx\d\zeta
\end{equation}
where $J=J_c-\tau$ and $J_c=\int_{\R^d}|\zeta_i|K_1(\zeta)\d\zeta$ is a critical constant such that for $J>J_c$ minimizers of $\bar\Fcal_{J,L}$ are trivial (i.e. $E=\emptyset$ or $E=\R^d$). 

For $\tau=J_c-J$ positive and small, the short range attractive term of perimeter type and the long range repulsive term with power law interaction kernel in~\eqref{eq:ef} enter in competition, and as a result one expects that global minimizers are one-dimensional and periodic (i.e., up to permutation of coordinates $E=\hat E\times\R^{d-1}$ with  $\hat E\subset\R$ periodic set). The fact that as $\tau\to0$ minimizers of~\eqref{eq:ef} are larger and larger stripes converging to trivial states with zero energy  suggests to rescale the functional by the optimal period of one-dimensional sets (of order $\tau^{-1/\beta}$) thus getting the functional~\eqref{eq:ftauel} for which optimal stripes have width and energy of order $O(1)$.

In particular, in more than one space dimensions the ground states retain less symmetries than the energy functional (in this case one-dimensionality vs. discrete symmetry w.r.t. coordinate permutations): such phenomenon is known as symmetry breaking, and together with nonlocality it makes the characterization of minimizers in several dimensions a very challenging problem.  

Defining 
\begin{equation*}
\mathcal C_L=\{E\subset\R^d:\,\text{ up to coordinate permuations } E=\hat E\times \R^{d-1} \text{ with }\hat E\subset\R\quad  \text{$L$-periodic}\}
\end{equation*} and
\begin{equation*}
	\Lambda(\tau)=\inf_L\inf_{E\in \mathcal C_L}\Fcal_{\tau,L}(E),
\end{equation*}
in~\cite{dr_arma} the authors proved that there exists $\bar{\tau}>0$ such that for every $0<\tau\leq\bar \tau$  the minimal energy $\Lambda(\tau)$ is attained on periodic stripes with density $1/2$ and period $2h^*_\tau$ for a unique $h^*_\tau>0$. By periodic stripes with density $1/2$ and period $2h^*_\tau$ (simply called periodic unions of stripes of period $2h^*_\tau$ in~\cite{dr_arma}) we mean sets which, up to permutations of coordinates and translations, are of the form
\begin{equation}\label{eq:stripes12}
	E=\bigcup_{k\in\Z}[2kh^*_\tau, (2k+1)h^*_\tau)\times\R^{d-1}.
\end{equation} 
Moreover and most importantly, periodic stripes of density $1/2$ and period $2h^*_\tau$ not only minimize $\Fcal_{\tau, L}$ among one-dimensional sets in the class $\mathcal C_L$ (when $L=2kh^*_\tau$), but they are indeed global minimizers of $\Fcal_{\tau,L}$ among all locally finite perimeter sets. 

Hence, for every $L=2kh^*_\tau$ and $0<\tau\leq\bar{\tau}$ with $\bar{\tau}$ independent of $k$, minimizers of $\Fcal_{\tau, L}$ have the form~\eqref{eq:stripes12}  and hence belong to the set
\begin{equation*}
	\mathcal C_{L,\alpha}:=\mathcal C_L\cap\Bigl\{\frac{|\hat E\cap[0,L)|}{L}=\alpha\Bigr\},
\end{equation*}
with $\alpha=1/2$.

An important question in the applications is to determine the structure of the ground states while imposing a volume constraint. Indeed, the mutual density of the two pure phases in chemical solutions or the total magnetization in physical systems is usually fixed and the ground states emerge in the same density class.

From a mathematical point of view we are no longer considering global minimizers as in \cite{gs_cmp,gr,dr_arma} but minimizers in an affine subspace. Thus, all the optimization estimates used in \cite{gs_cmp, dr_arma} are no longer valid.

  In order to state our results precisely, we say that a set $E$ is composed of periodic stripes with density $\alpha\in(0,1)$ and period $2h^*_{\tau,\alpha}$ if, up to translations and coordinate permutations, is of the form
\begin{equation}
	\label{eq:stripesalpha}
	E=E_{h^*_{\tau,\alpha},\alpha}\times\R^{d-1}=\bigcup_{k\in\Z}[2kh^*_{\tau,\alpha}, (2k+2\alpha)h^*_{\tau,\alpha})\times\R^{d-1}.
\end{equation}
In other words, the slices of such a one-dimensional set $E$ in the direction orthogonal to its boundary (e.g. $e_1$ for the set in~\eqref{eq:stripesalpha}) are \emph{simple periodic with density $\alpha$}.

 Our first result is the following.

\begin{theorem}\label{thm:1d}
	There exists $\hat{\tau}>0$ such that for all $0<\tau\leq\hat{\tau}$ and for all $\alpha\in(0,1)$  there exists a unique $h^*_{\tau,\alpha}>0$ such that the value 
	\begin{equation*}
		\Lambda(\tau,\alpha)=\inf_L\inf_{E\in \mathcal C_{L,\alpha}}\Fcal_{\tau,L}(E)
	\end{equation*}
 is attained on periodic stripes $E_{h^*_{\tau,\alpha},\alpha}\times\R^{d-1}$ of density $\alpha$ and period $2h^*_{\tau,\alpha}$. Moreover, 
 \begin{equation}\label{eq:stripeslarge}
 	\lim_{\alpha\to0}h^*_{\tau,\alpha}=+\infty, \qquad	\lim_{\alpha\to1}h^*_{\tau,\alpha}=+\infty,
 \end{equation}
namely the optimal stripes' period diverges as the volume density of the stripes tends to $0$.
\end{theorem}

In the above one-dimensional optimization, the structure of minimizers with volume constraint $\alpha$ (namely simple periodicity with a possibly non-unique optimal period) was already given for a general class of reflection positive functionals (and their finite perturbations) in~\cite{giuliani2009modulated}, combining ideas  from \cite{muller1993singular, alberti2001new,ren2003energy,chen2005periodicity,giuliani06_ising_model_with_long_range} and \cite{giuliani08_period_minim_local_mean_field_theor}. Here we prove uniqueness of the optimal period and its behaviour in the low density limit
for a class of reflection positive power law kernels.

Our main theorem concerning minimizers with preassigned volume constraint  in general dimension is the following. 
\begin{theorem}\label{thm:main} Let $\bar \alpha\in(0,1/2]$. Then there  exists $\bar{\tau}>0$ such that for all $0<\tau\leq\bar{\tau}$, $\alpha\in[\bar{\alpha},1-\bar{\alpha}]$ and $L=2kh^*_{\tau,\alpha}$, minimizers of $\Fcal_{\tau, L}$ satisfying the volume constraint $\frac{1}{L^d}\int_{Q_L}\chi_E(x)\dx=\alpha$ are periodic stripes with density $\alpha$ and period $2h^*_{\tau,\alpha}$. 
\end{theorem}

In particular, simple periodicity of minimizers is valid in the large volume limit (i.e. analogue of the thermodynamic limit at zero temperature).

\subsection{Scientific context} 

Pattern formation is ubiquitous in nature and emerges at nanoscale level in several physical/chemical systems. Despite the diverse interactions involved in such models, surprisingly similar patterns among which droplets or stripes/lamellae appear. As first suggested in~\cite{sa}, ground states given by periodic regular structures are universally believed to stem from the competition between short range attractive and long range repulsive (SALR) interactions. Although patterns are observed in experiments and produced by simulations, a rigorous mathematical proof of pattern formation starting from (discretely or continuously) symmetric functionals and domains is available only in a very few cases. The main difficulties lie in the symmetry breaking phenomenon coupled with nonlocality.

In the one-dimensional setting, simple periodicity of global minimizers is known to hold for convex or reflection positive repulsive kernels using respectively convexity arguments~\cite{hubbard1978generalized, pokrovsky1978properties,kerimov1999uniqueness,muller1993singular,ren2003energy, chen2005periodicity} or the reflection positivity technique~\cite{giuliani06_ising_model_with_long_range,giuliani08_period_minim_local_mean_field_theor,giuliani2009modulated}.  In particular, in~\cite{ren2003energy, chen2005periodicity, giuliani2009modulated} simple periodicity of minimizers is shown under arbitrary nontrivial volume constraint (see Theorem~\ref{thm:giu} for the results in \cite{giuliani2009modulated}).

In several space dimensions, there are only a few results were symmetry breaking and pattern formation are shown, whenever the broken symmetry is not enforced by explicit geometric quantities in the functionals or by asymmetries in the domain. The first characterization of ground states as periodic stripes in the discrete setting was given in~\cite{gs_cmp} for a discrete version of the functional~\eqref{eq:ef} in the range of exponents $p>2d$. In the continuous setting, breaking of discrete symmetry (i.e., symmetry w.r.t. permutation of coordinates) for the functional~\eqref{eq:ef} in the range of exponents $p>2d$ has been shown in~\cite{gr}. Exact periodicity and structure of minimizers for the wider range of exponents $p\geq d+2$ was achieved in~\cite{dr_arma}. In~\cite{gr,dr_arma} new quantities and techniques were introduced in order to deal with the richer geometry of locally finite perimeter sets in $\R^d$ and with defects carrying infinitesimal energy contributions.   In~\cite{ker} the above-mentioned results of~\cite{dr_arma} have been extended to a small range of exponents below $d+2$. Indeed, the interest in lowering the exponent of the kernel comes both from the mathematical challenges deriving from an increased nonlocality and from physical applications ($p=d+1$ corresponds to models for thin ferromagnatic films and Langmuir monolayers, $p=d$ to micromagnetics and $p=d-2$ to the Ohta-Kawasaki model for diblock copolymers~\cite{OhKa}). A characterization of ground states as periodic stripes has been shown also for repulsive kernels of screened-Coulomb (or Yukawa) type in~\cite{dr_siam} (see also~\cite{daneri2019symmetry}). In~\cite{dkr} one-dimensionality and periodicity of minimizers was shown for the diffuse interface version of the functional~\eqref{eq:ef}  in a finite periodic box. In~\cite{dr_therm} the results in~\cite{dkr} have been proved to hold in the large volume limit, analogue of the thermodynamic limit at zero temperature.

All such results regard global minimizers and ground states, which turn out to have volume density $1/2$ (with the convention that the two pure phases of the material correspond to the values $\{0,1\}$). Up to our knowledge Theorem~\ref{thm:main} is the first result showing discrete symmetry breaking and pattern formation in general dimension under preassigned volume constraint. The interesting phenomenon, in particular, is that even in the low density regime, in the periodic setting, stripes are the first type of pattern to emerge from the competition between attractive/repulsive forces in the range below the critical constant $J_c$.  Moreover, as shown by~\eqref{eq:stripeslarge}, stripes do not become finer and finer as the volume density decreases but the period of simple periodicity of the stripes enlarges more and more. Theorem~\ref{thm:1d} is achieved by explicitly computing $\Lambda(0,\alpha)$ and $h^*_{0,\alpha}$ and showing uniform  convexity properties of the function $\alpha\mapsto\Lambda(\tau,\alpha)$ and uniform continuity properties w.r.t. $\tau$ as $\tau\to0$ (see the proof of point $(i)$ in Theorem \ref{thm:convex}).

 Regarding further fields of interest for pattern formation under attractive/repulsive forces in competition, we mention the following. Evolution problems of gradient flow type related to functionals with attractive/repulsive nonlocal terms in competition, both in presence and in absence of diffusion, are also well studied (see e.g.~\cite{carrillo2014derivation, carrillo2019blob, carrillo2011global, craig2020aggregation, craig2017nonconvex}). In particular, one would like to show convergence of the gradient flows or of their deterministic particle  approximations to configurations which are periodic or close to periodic states.
Another interesting direction would be to extend our rigidity results to non-flat surfaces without interpenetration of matter as investigated for rod and plate theories  in~\cite{kupferman2014riemannian, lewicka2010shell, olbermann2017interpenetration}.

\subsection{General strategy}

The overall strategy to prove Theorem~\ref{thm:main} consists very roughly in the following.  As in~\cite{dr_arma,dr_siam,dr_therm} a smaller scale $l<L$ is introduced and the set $[0,L)^d$ is partitioned into $d+1$ sets, where the first $d$ sets $\{A_i\}_{i=1}^d$ are constituted of points $z$ which on the cube $Q_l(z)$ of centre $z$ and sidelength $l$ are $L^1$-close to stripes with boundaries orthogonal to $e_i$, $i\in\{1,\dots,d\}$. Calling $A=[0,L)^d\setminus \cup_{i=1}^dA_i$, the aim is to prove that $A=\emptyset$. This indeed, by continuity of the distance function from stripes with boundary in a given direction, would show that $[0,L)^d=A_i$ for some $i\in\{1,\dots,d\}$, and once discrete symmetry is broken, namely the minimizer $E$ is close to stripes with boundaries orthogonal to $e_i$, a stability argument allows to conclude that $E$ is one-dimensional. In the end, the fact that $L$ is a multiple of the optimal period $2h^*_{\tau,\alpha}$ for stripes with density $\alpha$ allows to apply the one-dimensional optimization result of Theorem~\ref{thm:1d} and deduce simple periodicity of the optimal one-dimensional profiles.

 The main difficulty w.r.t.~\cite{dr_arma,dr_siam,dr_therm} is that whenever we consider one-dimensional optimization arguments on slices of the set $A_i$ (namely on the set where $E$ is close to stripes with boundaries orthogonal to $e_i$) we cannot bound the energy of such slices from below with the minimal energy density $\Lambda(\tau)\frac{|A_i|}{L^d}=\Lambda(\tau,1/2)\frac{|A_i|}{L^d}$. This would lead indeed to a coarse global lower bound when restricting to sets with volume constraint $\alpha$, due to the fact that $\Lambda(\tau,1/2)<\Lambda(\tau,\alpha)$ whenever $\alpha\neq\frac12$.  Hence, on each connected $i$-slice $J$ of $A_i$ one has to take into account instead the volume density $\alpha(J)$ of an $l$-neighbourhood of $J$, obtaining a lower bound with variable energy density $\Lambda(\tau,\alpha(J))$. The difficulty related to the fact that the local volume densities $\alpha(J)$ are not controllable by the global volume constraint $\alpha$ on $[0,L)^d$ is overcome thanks to strict convexity properties of the function $\alpha\mapsto\Lambda(\tau,\alpha)$, which we will prove in Section~\ref{sec:One-Dimensional}, Theorem \ref{thm:convex}. The role of strict convexity will be to penalize oscillations of volume densities on the different connected components of the sets of the partition $A\cup A_1\cup\dots\cup A_d$ with respect to the global volume density $\alpha$. Another difficulty is due to the role of the boundary between the different $A_i$ in the estimates, which differently from \cite{dr_arma} has to be considered separately resulting in a different partition.

\section{Notation and preliminary results}

Let $d \geq 1$. On $\R^d$ let us denote by $\langle\cdot,\cdot\rangle$ the Euclidean scalar product and by $|\cdot|$ the Euclidean norm. Let $e_1, \dots, e_n$ be the canonical basis on $\R^d$. We will often employ slicing arguments, for this reason we need definitions concerning the $i$-th component. For $x \in \R^d$ let $x_i =\langle x,e_i\rangle $ and $x_i^{\perp} := x - x_ie_i$.  
Let 
$\|x\|_1=\sum_{i=1}^d|x_i|$ be the $1$-norm  and $\|x\|_\infty=\max_i|x_i|$ the $\infty$-norm. 
While writing slicing formulas,
with a slight abuse of notation we will sometimes identify 
$x_i\in[0,L)$  with the point $x_ie_i\in[0,L)^d$
and $\{x_i^{\perp}:\,x\in[0,L)^d\}$ with 
$[0,L)^{d-1}\subset\R^{d-1}$ 
so that  $x_i^{\perp} \in [0,L)^{d-1}$. 

Whenever $\Omega\subset\R^d$ is a measurable set, we denote by $\mathcal H^{d-1}(\Omega)$ its $(d-1)$-dimensional Hausdorff measure and by $|\Omega|$ its Lebesgue measure.

Given a measure $\mu$ on $\R^d$, we denote by $|\mu|$ its total variation. 

We recall that  a set $E\subset\R^d$ is of (locally) finite perimeter if the distributional derivative of its characteristic function $\chi_E$ is a (locally) finite measure. We denote by $\partial E$ the reduced boundary of $E$ and by $\nu^E$ its exterior normal. 

The anisotropic $1$-perimeter of $E$ is given by 
\[
\per_1(E,[0,L)^d):=\int_{\partial E\cap [0,L)^d}\|\nu^E(x)\|_1\d\mathcal H^{d-1}(x)
\]
and, for $i\in\{1,\dots,d\}$ 
\begin{equation}
	\label{eq:perI}
	\per_{1i}(E,[0,L)^d)=\int_{\partial E\cap [0,L)^d}|\nu^E_i(x)|\d\mathcal H^{d-1}(x),
\end{equation}
thus $\per_1(E,[0,L)^d)=\sum_{i=1}^d\per_{1i}(E,[0,L)^d)$. 

For $i\in\{1,\dots,d\}$,  we define the one-dimensional slices of $E\subset\R^d$ in direction $e_i$ by

\[
E_{x_i^\perp}:=\bigl\{s\in[0,L):\,se_i+ x_i^\perp\in E\bigr\}.
\]

Whenever $E$ is a set of locally finite perimeter, for a.e. $x_i^\perp$ its slice $E_{x_i^\perp}$ is a set of locally finite perimeter in $\R$ and the following slicing formula  holds for every $i\in\{1,\dots,d\}$
\[
\per_{1i}(E,[0,L)^d)=\int_{\partial E\cap [0,L)^d}|\nu^E_i(x)|\d\mathcal H^{d-1}(x)=\int_{[0,L)^{d-1}}\per_1(E_{x_i^\perp},[0,L))\dx_i^\perp.
\]

Consider  $E\subset\R$  a set of locally finite perimeter and $s\in\partial E$ a point in the relative boundary of $E$. We will denote by 
\begin{equation}
	\label{eq:s+s-}
	\begin{split}
		s^+ &:= \inf\{ t' \in \partial E, \text{with } t' > s  \} \\ s^- &:= \sup\{ t' \in \partial E, \text{with } t' < s  \}. 
	\end{split}
\end{equation}

We will also apply slicing on small cubes, depending on $l$, around a point. Therefore we introduce the following notation.
For $r> 0$ and $x^{\perp}_i$ we let $Q_{r}^{\perp}(x^\perp_{i}) = \{z^\perp_{i}:\, \|x^{\perp}_{i} - z^{\perp}_{i} \|_\infty \leq r  \}$ or we think of $x_i^\perp\in[0,L)^{d-1}$ and $Q_r^\perp(x_i^\perp)$ as a subset of $\R^{d-1}$. 
We denote also by $Q^i_r(t_i)\subset\R$ the interval of length $r$ centred in $t_i$.

From~\cite{gr,dr_arma} we recall that, using the equality $|\chi_E(x)-\chi_E(x+\zeta)|=|\chi_E(x)-\chi_E(x+\zeta_ie_i)|+|\chi_E(x+\zeta_ie_i)-\chi_E(x+\zeta)|-2|\chi_E(x)-\chi_E(x+\zeta_ie_i)||\chi_E(x+\zeta_ie_i)-\chi_E(x+\zeta)|$ and $Q_L$-periodicity,  the following lower bound holds.

\begin{align}
	\label{eq:gstr1}
	\Fcal_{\tau,L} (E) 
	&\geq-\frac{1}{L^d}\sum_{i=1}^{d}\per_{1i}(E,[0,L)^d) + \frac{1}{L^d}\sum_{i=1}^{d} \Big[\int_{[0,L)^d\cap \partial E} \int_{\R^d} |\nu^{E}_{i} (x)| |\zeta_{i} | K_{\tau}(\zeta) \d\zeta \d\mathcal H^{d-1}(x) \notag\\ 
	&- \int_{[0,L)^d } \int_{\R^d} |\chi_{E}(x + \zeta_ie_i)  - \chi_{E}(x) | K_{\tau}(\zeta) \d\zeta\dx \Big] \notag\\ 
	&+ \frac{2}{d} \frac{1}{L^d}\sum_{i=1}^d \int_{[0,L)^d } \int_{\R^d} |\chi_{E}(x + \zeta_{i}e_i) -\chi_{E}(x) | | \chi_{E}(x + \zeta^{\perp}_i) - \chi_{E}(x) | K_{\tau}(\zeta)
	\d\zeta \dx.
\end{align}

Notice that in~\eqref{eq:gstr1} equality holds whenever the set $E$ is a union of stripes.  Thus, proving that optimal unions of stripes with density $\alpha$ are the minimizers of the r.h.s. of~\eqref{eq:gstr1} in the set $\mathcal C_{L,\alpha}$ implies that they are the minimizers for $\Fcal_{\tau,L}$.

Let us define
\[
\widehat K_{\tau}(\zeta_i)=\int_{ \R^{d-1}}K_{\tau}(\zeta_ie_i+\zeta_i^\perp)\d\zeta_i^\perp.
\]

As in Section 7 of~\cite{dr_arma} we further decompose the r.h.s. of~\eqref{eq:gstr1} as follows.

\begin{equation}\label{eq:decomp_r}
	\begin{split}
		&-\frac{1}{L^d}\per_{1i}(E,[0,L)^d)+ \frac{1}{L^d}\Big[\int_{[0,L)^d\cap \partial E} \int_{\R^d} |\nu^{E}_{i} (x)| |\zeta_{i} | K_{\tau}(\zeta)\d\zeta\d\mathcal H^{d-1}(x) \\& - \int_{[0,L)^d } \int_{\R^d} |\chi_{E}(x + \zeta_ie_i)  - \chi_{E}(x) | K_{\tau}(\zeta) \d\zeta \dx \Big]
		= \frac{1}{L^d}\int_{[0,L)^{d-1}}  \sum_{s\in  \partial E_{t^{\perp}_{i}}\cap [0,L]} r_{i,\tau}(E,t_{i}^{\perp},s)
		\dt^{\perp}_i, 
	\end{split}
\end{equation}
where  for $s\in \partial E_{t_{i}^{\perp}}$ 
\begin{equation}\label{eq:ritau}
	\begin{split}
		r_{i,\tau}(E, t_{i}^{\perp},s) := -1 + \int_{\R} |\zeta_{i}| \widehat K_\tau (\zeta_{i})\d\zeta_i &- \int_{s^-}^{s}\int_{0}^{+\infty} |\chi_{E_{t_{i}^{\perp}}}(u + \rho) - \chi_{E_{t_{i}^{\perp}}}(u) | \widehat K_\tau (\rho)\d\rho \du  \\ & - \int_{s}^{s^+}\int_{-\infty}^{0} |\chi_{E_{t_{i}^{\perp}}}(u + \rho) - \chi_{E_{t_{i}^{\perp}}}(u) |\widehat K_\tau (\rho) \d\rho \du.\\ 
	\end{split}
\end{equation}
and $s^-<s<s^+$ are as in~\eqref{eq:s+s-}.

Defining
\begin{equation} 
	\label{eq:defFE}
	\begin{split}
		f_{E}(t_{i}^{\perp},t_{i},\zeta_{i}^{\perp},\zeta_{i}): =  |\chi_{E}(t_ie_i+t_i^\perp + \zeta_{i}e_i )  - \chi_{E}(t_ie_i+t_i^\perp)| |\chi_{E}(t_ie_i+t_i^\perp + \zeta^{\perp}_{i}) - \chi_{E}(t_ie_i+t_i^\perp) |,
	\end{split}
\end{equation} 
one has that
\begin{align}
	\label{eq:decomp_double_prod}
	\frac{2}{d} \frac{1}{L^d}\sum_{i=1}^d &\int_{[0,L)^d } \int_{\R^d} |\chi_{E}(x + \zeta_{i}e_i) -\chi_{E}(x) | | \chi_{E}(x + \zeta^{\perp}_i) - \chi_{E}(x) | K_{\tau}(\zeta)
	\d\zeta \dx=\notag\\
	&=\frac{2}{d}\frac{1}{L^d} \int_{[0,L)^d } \int_{\R^d} f_{E}(t_{i}^{\perp},t_{i},\zeta_{i}^{\perp},\zeta_{i}) K_{\tau}(\zeta)  \d\zeta \dt\notag\\
	& =\frac{1}{L^d}  \int_{[0,L)^{d-1}} \sum_{s\in \partial E_{t_{i}^{\perp}}\cap [0,L]} v_{i,\tau}(E,t_{i}^{\perp},s)\dt_{i}^\perp + \frac{1}{L^d}\int_{[0,L)^d} w_{i,\tau}(E,t_i^\perp,t_i) \dt,
\end{align}
where
\begin{equation}\label{eq:witau}
	{w}_{i,\tau}(E,t_{i}^{\perp},t_{i}) = \frac{1}{d}\int_{\R^d}  
	f_{E}(t_{i}^{\perp},t_{i},\zeta_{i}^{\perp},\zeta_{i}) K_\tau(\zeta)  \d\zeta. 
\end{equation}
and
\begin{equation}\label{eq:vitau}
	v_{i,\tau}(E,t_{i}^{\perp},s) =  \frac{1}{2d}\int_{s^{-}}^{s^{+}} \int_{\R^{d}} f_{E}(t^{\perp}_{i},u,\zeta^{\perp}_{i},\zeta_{i}) K_\tau(\zeta) \d\zeta\du.
\end{equation}

Hence, putting together~\eqref{eq:decomp_r} and~\eqref{eq:decomp_double_prod} one has the following decomposition

\begin{align}
	\label{eq:decomposition}
	\Fcal_{\tau,L} (E) 
	&\geq \frac{1}{L^d}\int_{[0,L)^{d-1}}  \sum_{s\in  \partial E_{t^{\perp}_{i}}\cap [0,L]} r_{i,\tau}(E,t_{i}^{\perp},s)
	\dt^{\perp}_i\notag\\
	&+\frac{1}{L^d}  \int_{[0,L)^{d-1}} \sum_{s\in \partial E_{t_{i}^{\perp}}\cap [0,L]} v_{i,\tau}(E,t_{i}^{\perp},s)\dt_{i}^\perp \notag\\
	&+ \frac{1}{L^d}\int_{[0,L)^d} w_{i,\tau}(E,t_i^\perp,t_i) \dt.
\end{align}

The term $r_{i,\tau}$ penalizes oscillations with high frequency  in direction $e_i$, namely sets $E$ whose slices in direction $e_i$ have boundary points at small minimal distance (see Lemma~\ref{rmk:stimax1}). The term $v_{i,\tau}$  penalizes oscillations in direction $e_i$ whenever the neighbourhood of the point in $\partial E\cap{Q_l(z)}$ is close in $L^1$ to a stripe oriented along $e_j$ (see Proposition~\ref{lemma:stimaContributoVariazionePiccola}).

For every cube  $Q_{l}(z)$, with $l<L$ and $z\in[0,L)^d$, define now the following localization of $\Fcal_{\tau,L}$
\begin{equation} 
	\label{eq:fbartau}
	\begin{split}
		\bar{F}_{i,\tau}(E,Q_{l}(z)) &:= \frac{1}{l^d  }\Big[\int_{Q^{\perp}_{l}(z_{i}^{\perp})} \sum_{\substack{s \in \partial E_{t_{i}^{\perp}}\\ t_{i}^{\perp}+se_i\in Q_{l}(z)}} (v_{i,\tau}(E,t_{i}^{\perp},s)+ r_{i,\tau}(E,t_{i}^{\perp},s)) \dt_{i}^{\perp} + \int_{Q_{l}(z)} {w_{i,\tau}(E,t_{i}^{\perp}, t_i) \dt}\Big],\\
		\bar{F}_{\tau}(E,Q_{l}(z)) &:= \sum_{i=1}^d\bar F_{i,\tau}(E,Q_{l}(z)).
	\end{split}
\end{equation} 

The following inequality holds:
\begin{equation}
	\label{eq:gstr14}
	\begin{split}
		\Fcal_{\tau,L}(E) \geq \frac{1}{L^d} \int_{[0,L)^d}  
		\bar{F}_{\tau}(E,Q_{l}(z)) \dz. 
	\end{split}
\end{equation}
Since in~\eqref{eq:gstr14} equality holds for unions of stripes, in order to prove Theorem~\ref{thm:main} one can reduce to show that the minimizers of its right hand side in the class $\mathcal C_{L,\alpha}$ are periodic optimal stripes  of density $\alpha$ provided $\tau$ and $L$ satisfy the conditions of the theorem.

In the next definition we define a quantity which measures the $L^1$ distance of a set from being a union of stripes.

\begin{definition}
	\label{def:defDEta}
	For every $\eta$ we denote by $\Acal^{i}_{\eta}$ the family of all sets $F$ such that  
	\begin{enumerate}[(i)]
		\item they are union of stripes oriented along the direction $e_i$ 
		\item their connected components of the boundary are distant at least $\eta$. 
	\end{enumerate}
	We denote by 
	\begin{equation} 
		\label{eq:defDEta}
		\begin{split}
			D^{i}_{\eta}(E,Q) := \inf\Big\{ \frac{1}{\vol(Q)} \int_{Q} |\chi_{E} -\chi_{F}|:\ F\in \Acal^{i}_{\eta} \Big\} \quad\text{and}\quad D_{\eta}(E,Q) = \inf_{i} D^{i}_{\eta}(E,Q).
		\end{split}
	\end{equation} 
	Finally, we let $\mathcal A_\eta:=\cup_{i}\mathcal A^i_{\eta}$.
\end{definition}

We recall also the following properties of the functional defined in~\eqref{eq:defDEta}.

\begin{remark} The distance function from the set of stripes satisfies the following properties.
	\label{rmk:lip} \ 
	\begin{enumerate}[(i)]
		\item Let $E\subset\R^d$.  Then the map $z\mapsto D_{\eta}(E,Q_{l}(z))$ is Lipschitz, with Lipschitz constant $C_d/l$, where $C_d$ is a constant depending only on the dimension $d$. 
		
		In particular, whenever $D_{\eta}(E,Q_{l}(z)) > \alpha$ and $D_{\eta}(E,Q_{l}(z')) < \beta$,  then $|z - z'|> l(\alpha - \beta)/C_{d}$.

		\item
		For every $\varepsilon$ there exists ${\delta}_0= \delta_0(\varepsilon)$ such that  for every $\delta \leq \delta_0 $ whenever $D^{j}_{\eta}(E,Q_{l}(z))\leq \delta$ and $D^{i}_{\eta}(E,Q_{l}(z))\leq \delta$ with $i\neq j$ for some $\eta>0$,  it holds 
		\begin{equation*}
			\begin{split}
				\min\big(|Q_l(z)\setminus E|, |E \cap Q_l(z)| \big) \leq\varepsilon. 
			\end{split}
		\end{equation*}
	\end{enumerate}
\end{remark}

\section{The one-dimensional problem}
\label{sec:One-Dimensional}

In this section we study properties of the energy functional  $\Fcal_{\tau, L}$ on one-dimensional sets satisfying the volume constraint $\alpha\in(0,1)$, namely sets belonging to the set $\mathcal C_{L,\alpha}$.



First we  recall the following result, which is a corollary of Theorem 1 in \cite{giuliani2009modulated}. In \cite{giuliani2009modulated} the result is proved for general reflection positive kernels and their finite perturbations, while here we focus on the reflection positive kernel $K_\tau$.

\begin{theorem}
	\label{thm:giu}
Let $0<\alpha<1$, $\tau\geq0$. 
 Then, the value 
	\begin{equation*}
		\Lambda(\tau,\alpha)=\inf_L\inf_{E\in \mathcal C_{L,\alpha}}\Fcal_{\tau,L}(E)
	\end{equation*}
	is attained by the functional $\Fcal_{\tau, 2kh^*_{\tau,\alpha}}$ on periodic stripes $E_{h^*_{\tau,\alpha},\alpha}$ of density $\alpha$ and period $2h^*_{\tau,\alpha}$ for some (possibly non-unique) $h^*_{\tau,\alpha}>0$ and for all $k\in\N$. 
\end{theorem}


The main result of this section is the following

\begin{theorem}\label{thm:convex}
	There exists $\hat\tau>0$ such that for all $0\leq\tau\leq\hat\tau$ the following holds. 
	\begin{enumerate}
		\item[(i)] For every $\alpha\in(0,1)$ there exists a unique $h^*_{\tau,\alpha}>0$ such that 
		\begin{equation*}
			\Lambda(\tau,\alpha)=\Fcal_{\tau,2h^*_{\tau,\alpha}}(E_{h^*_{\tau,\alpha},\alpha})=\inf_h\Fcal_{\tau,2h}(E_{h,\alpha}).
		\end{equation*}
	Moreover, 
	\[
	\lim_{\alpha\to0}h^*_{\tau,\alpha}=+\infty,\qquad 	\lim_{\alpha\to1}h^*_{\tau,\alpha}=+\infty.
	\]
		\item[(ii)] There exists $C_2>0$ such that
		\begin{equation}\label{eq:estderlam}
			\sup_{\alpha\in(0,1)}|\partial_\alpha\Lambda(\tau,\alpha)|\leq C_2.
		\end{equation}
		\item[(iii)] The map $(0,1)\ni\alpha\mapsto\Lambda(\tau,\alpha)$ is strictly convex with
		\begin{equation}\label{eq:2derineq}
			\partial^2_\alpha\Lambda(\tau,\alpha)\geq \tilde c\min\{\alpha^{q-1},(1-\alpha)^{q-1}\}
		\end{equation}
	for some constant $\tilde c>0$ independent of $\alpha$. 
		\end{enumerate}
\end{theorem}

Theorem \ref{thm:1d} corresponds to point $(i)$ of Theorem \ref{thm:convex}. Due to the symmetry of the problem w.r.t. $\alpha=1/2$, w.l.o.g. we will consider in the following densities $\alpha\in(0,1/2]$.

Let us start with the following preliminary lemma.

\begin{lemma}
	\label{lemma:econv0}
	Let $\alpha\in\Big(0,\frac12\Big]$. Then, setting for simplicity $F_{\tau}(h,\alpha)=\Fcal_{\tau,2h}(E_{h,\alpha})$, one has that 
	\begin{equation}\label{eq:ftauform}
		F_{\tau}(h,\alpha)=-\frac1h+\frac{C(q,\alpha,\tau^{1/\beta}/h)}{h^{q-1}}, 
	\end{equation}
where for $s>0$
\begin{equation}\label{eq:formtau}
	C(q,\alpha,s)=\frac{2C_1}{(q-1)(q-2)}\Big\{\sum_{k\geq 0}\frac{1}{(2k+2\alpha+s)^{q-2}}+\frac{1}{(2k+2-2\alpha+s)^{q-2}}-\frac{2}{(2k+2+s)^{q-2}}\Big\}
\end{equation}
and $C_1=\int_{\R^{d-1}}\frac{1}{(\|\xi\|_1+1)^p}\d\xi$.
\end{lemma}

\begin{proof}
Setting $E_{h,\alpha}=\hat E_{h,\alpha}\times\R^{d-1}$, recall that 
\begin{align*}
&	\Fcal_{\tau,2h}(E_{h,\alpha})=\notag\\
	&=\frac{1}{2h}\Bigl(-\mathrm{Per}(\hat E_{h,\alpha},[0,2h))+\int_{ \R}\widehat{K}_{\tau}(z)\Bigl[\mathrm{Per}(\hat E_{h,\alpha},[0,2h))|z|-\int_0^{2h}|\chi_{\hat E_{h,\alpha}}(x)-\chi_{\hat E_{h,\alpha}}(x+z)|\dx\Bigr]\dz\Bigr)\notag\\
	&=\frac{1}{2h}\Bigl(-\mathrm{Per}(\hat E_{2h,\alpha},[0,2h))+C_{1}\int_{ \R}\frac{1}{(|z|+\tau^{1/\beta})^{q}}\Bigl[\mathrm{Per}(\hat E_{h,\alpha},[0,2h))|z|-\int_0^{2h}|\chi_{\hat E_{h,\alpha}}(x)-\chi_{\hat E_{h,\alpha}}(x+z)|\dx\Bigr]\dz\Bigr),
\end{align*}
where $C_1=\int_{\R^{d-1}}\frac{1}{(\|\xi\|_1+1)^p}\d\xi$ is such that
\[
\widehat K_{\tau}(z)=C_1\frac{1}{(|z|+\tau^{1/\beta})^q}, \quad q=p-d+1.
\]

Setting $A_{\tau,h,\alpha}=\Fcal^{1d}_{\tau,2h}(\hat E_{h,\alpha})+\frac{1}{h}$ and by simple computations one obtains
\begin{align}
	A_{\tau,h,\alpha}=&\frac{C_1}{2h}\int_{\R}\frac{1}{|z+\tau^{1/\beta}|^{q}} \Big(|z|-\int_0^{2\alpha h} \chi_{\hat E_{h,\alpha}^c}(x+z)\dx\Big)\dz \notag \\ & + \frac{C_1}{2h}\int_{\R}\frac{1}{|z+\tau^{1/\beta}|^{q}} \lt(|z|-\int_{2\alpha h}^{2h} \chi_{\hat E_{h,\alpha}}(x+z)\dx\rt)\dz\notag\\
	=& \frac{C_1}{h} \int_{\R^+}(z+\tau^{1/\beta})^{-q} \lt(z-\int_0^{2\alpha h} \chi_{\hat E_{h,\alpha}^c}(x+z)\dx\rt)\dz \notag\\ &+\frac{C_1}{h} \int_{\R^+}(z+\tau^{1/\beta})^{-q} \lt(z-\int_{2\alpha h}^{2h} \chi_{\hat E_{h,\alpha}}(x+z)\dx\rt)\dz.\label{eq:aform}
\end{align}

Using that 
\begin{align*}
	\chi_{\hat E_{h,\alpha}^c}\cdot\chi_{[0,\infty)}&=\chi_{[0,2\alpha h]^c}-\sum_{k\geq 1}\chi_{[2kh,(2k+2\alpha )h)},\notag\\
	\chi_{\hat E_{h,\alpha}}\cdot\chi_{[2\alpha h,\infty)}	&=\chi_{[2\alpha h,2h]^c}-\sum_{k\geq 1}\chi_{[(2k+2\alpha )h,(2k+2)h)},
\end{align*}
one has that
\begin{align}
	\int_{\R^+}(z+\tau^{1/\beta})^{-q} &\lt(z-\int_0^{2\alpha h} \chi_{\hat E_{h,\alpha}^c}(x+z)\dx\rt)\dz=\notag\\
	=&\int_{2\alpha h}^{+\infty}(z+\tau^{1/\beta})^{-q}(z-2\alpha h)\dz+\sum_{k\geq 1}\int_{2kh}^{(2k+2\alpha )h}\int_0^{2\alpha h}\frac{1}{(|y-x|+\tau^{1/\beta})^q}\dx\dy\notag\\
	=&\frac{(2\alpha h+\tau^{1/\beta})^{-(q-2)}}{(q-2)(q-1)}+\frac{1}{(q-2)(q-1)}\sum_{k\geq 1}\Big[\frac{1}{((2k-2\alpha)h+\tau^{1/\beta} )^{q-2}}\notag\\&+\frac{1}{((2k+2\alpha)h+\tau^{1/\beta} )^{q-2}} -\frac{2}{(2kh+\tau^{1/\beta})^{q-2}}\Big]\label{eq:dis1}
\end{align}
and analogously
\begin{align}
	\int_{\R^+}(z+\tau^{1/\beta})^{-q} &\lt(z-\int_{2\alpha h}^{2h} \chi_{\hat E_{h,\alpha}}(x+z)\dx\rt)\dz=\notag\\
	=&\int_{2(1-\alpha )h}^{+\infty}(z+\tau^{1/\beta})^{-q}(z-2(1-\alpha )h)\dz+\sum_{k\geq 1}\int_{(2k+2\alpha )h}^{(2k+2)h}\int_{2\alpha h}^{2h}\frac{1}{(|y-x|+\tau^{1/\beta})^q}\dx\dy\notag\\
	=&\frac{(2(1-\alpha )h+\tau^{1/\beta})^{-(q-2)}}{(q-2)(q-1)}+\frac{1}{(q-2)(q-1)}\sum_{k\geq 1}\Big[\frac{1}{((2k-2+2\alpha)h +\tau^{1/\beta})^{q-2}}\notag\\
	&+\frac{1}{((2k+2-2\alpha)h+\tau^{1/\beta} )^{q-2}}-\frac{2}{(2kh+\tau^{1/\beta})^{q-2}}\Big].\label{eq:dis2}
\end{align}
Inserting~\eqref{eq:dis1} and~\eqref{eq:dis2} in~\eqref{eq:aform} one obtains
\begin{align}
	F_{\tau}(h,\alpha)&=-\frac{1}{h}+\frac{C_1}{h(q-1)(q-2)}\Big\{(2\alpha h+\tau^{1/\beta})^{-(q-2)}+(2(1-\alpha)h+\tau^{1/\beta})^{-(q-2)}\notag\\
	&+\sum_{k\geq 1}\frac{1}{[(2k-2\alpha)h+\tau^{1/\beta}]^{q-2}}+\frac{1}{[(2k+2\alpha)h+\tau^{1/\beta}]^{q-2}}+\frac{1}{[(2k-2+2\alpha)h+\tau^{1/\beta}]^{q-2}}\notag\\
	&+\frac{1}{[(2k+2-2\alpha)h+\tau^{1/\beta}]^{q-2}}-\frac{4}{[2kh+\tau^{1/\beta}]^{q-2}}\Big\},\notag
\end{align}
which by rearranging the terms gives
\begin{align}
	F_{\tau}(h,\alpha)=&-\frac{1}{h}+\frac{C_1}{h(q-1)(q-2)}\Big\{\sum_{k\geq 0}\frac{2}{((2k+2\alpha)h+\tau^{1/\beta})^{q-2}}\notag\\
	&+\frac{2}{((2k+2-2\alpha)h+\tau^{1/\beta})^{q-2}}-\frac{4}{((2k+2)h+2\tau^{1/\beta})^{q-2}}\Big\},\notag	
\end{align}
corresponding to~\eqref{eq:formtau}.

\end{proof}

\begin{lemma}\label{lemma:estf0}
	The following holds.
	\begin{enumerate}
	\item[(i)] For every $\alpha\in(0,1/2]$, there exists a unique $h^*_{0,\alpha}$ such that $\Lambda(0,\alpha)=F_0(h^*_{0,\alpha},\alpha)$, given by
	\begin{equation}\label{eq:h0form}
		h^*_{0,\alpha}=\Big((q-1)C(q,\alpha,0)\Big)^{\frac{1}{q-2}}.
	\end{equation}
\item[(ii)] For every $0<\eps\ll1$ there exist $0<c_1<c_2$ such  that for all $\alpha\in(0,1/2]$ 
\begin{equation}\label{eq:f0min}
	F_0(h,\alpha)\leq F_0(h^*_{0,\alpha},\alpha)+\eps\quad\Rightarrow\quad\frac{c_1}{\alpha}\leq h\leq\frac{c_2}{\alpha},
\end{equation}
with $c_1/\alpha,\,c_2/\alpha\to h^*_{0,\alpha}$ as $\eps\to0$. 
\item[(iii)] There exist $0<\bar c_1<\bar c_2$ and $\tilde c_3>c_3>0$ such that for all $\alpha\in(0,1/2]$ it holds  $\bar c_1/\alpha<h^*_{0,\alpha}<\bar c_2/\alpha$ and  
 \begin{equation}\label{eq:d2f0}
	\frac{\bar c_1}{\alpha}\leq h\leq\frac{\bar c_2}{\alpha}\quad\Rightarrow  \quad\tilde c_3\alpha^3\geq\partial^2_h F_0(h,\alpha)\geq c_3\alpha^3.
\end{equation}
 \end{enumerate}

\end{lemma}

\begin{proof}
	The explicit formula~\eqref{eq:h0form} follows from an explicit calculation of the unique critical point of $h\mapsto F_0(h,\alpha)$ using the expression~\eqref{eq:ftauform} when $\tau=0$.

	By convexity of the function $\frac{1}{x^{q-2}}$ for $x>0$ applied to the summation in~\eqref{eq:formtau} one immediately gets that
	\begin{equation}\label{eq:cbound1}
		C(q,\alpha,0)\geq\frac{2}{(2\alpha)^{q-2}}\frac{C_1}{(q-1)(q-2)}.
	\end{equation}
	Estimating the summation in~\eqref{eq:formtau} from above one has that
	\begin{align}
		C(q,\alpha,0)&\leq \frac{C_1}{(q-1)(q-2)}\Big[\frac{2}{(2\alpha)^{q-2}}+\sum_{k\geq 1}\frac{C}{(2k)^{q-1}}\alpha\Big]\notag\\
		&\leq\frac{\bar C}{(2\alpha)^{q-2}},\label{eq:cbound2}
	\end{align}
	where we used that $q\geq3$ when $p\geq d+2$, hence the above series converges.

	Inserting the expression for $h^*_{0,\alpha}$ in the functional, one obtains
	\begin{equation*}
		F_0(h^*_{0,\alpha},\alpha)=-\frac{q-2}{(q-1)^{\frac{(q-1)}{(q-2)}}}C(q,\alpha,0)^{-\frac{1}{q-2}}.
	\end{equation*}
Inserting the explicit expression for $F_0(h,\alpha)$ and $F_0(h^*_{0,\alpha},\alpha)$, the left inequality in~\eqref{eq:f0min} is thus of the form
\begin{equation}\label{eq:ineqridotta}
	-h^{q-2}\leq -C(q,\alpha,0)-\big[\bar CC(q,\alpha,0)^{-1/(q-2)}-\eps\big]h^{q-1}.
\end{equation}
The upper and lower bound on $h$ in~\eqref{eq:f0min} follows then from the upper and lower bound for $C(q,\alpha,0)$ given in~\eqref{eq:cbound1} and~\eqref{eq:cbound2} (giving in particular that $C(q,\alpha,0)\sim\alpha^{-(q-2)}$) and the inequality~\eqref{eq:ineqridotta}.

Regarding~\eqref{eq:d2f0} one has that
\begin{align}
	\partial^2_h F_0(h,\alpha)=\frac{-2h^{q-2}+q(q-1)C(q,\alpha,0)}{h^{q+1}}.
\end{align}
In particular,
\begin{equation*}
	\partial^2_h F_0(h^*_{0,\alpha},\alpha)={(q-1)^{-3/(q-2)}(q-2)C(q,\alpha,0)^{-3/(q-2)}}.
\end{equation*}
Hence	by~\eqref{eq:cbound1} and~\eqref{eq:cbound2} and the fact that $q\geq3$ also~\eqref{eq:d2f0} holds.
\end{proof}

\begin{lemma}\label{lemma:tauest}
	There exist $c_4>0$ and $\tilde\tau>0$ such that for all $0\leq\tau\leq\tilde\tau$ and  $\alpha\in(0,1/2]$ one has that $c_4/\alpha<h^*_{0,\alpha}$ and  
	\begin{equation}\label{eq:c4}
		F_\tau(h,\alpha)<0\quad\Rightarrow\quad h\geq\frac{c_4}{\alpha}.
	\end{equation}
Moreover,  there exist $c_5,c_6,c_7>0$ such that for all $\alpha\in(0,1/2]$ and for all $h\geq c_4/\alpha$ it holds 
	\begin{align}
		\big|F_\tau(h,\alpha)-F_0(h,\alpha)\big|&\leq c_5\tau^{1/\beta}\alpha\label{eq:d0}\\
		\big|\partial_hF_\tau(h,\alpha)-\partial_hF_0(h,\alpha)\big|&\leq c_6\tau^{1/\beta}\alpha^2,\label{eq:d1}\\
		\big|\partial^2_hF_\tau(h,\alpha)-\partial^2_hF_0(h,\alpha)\big|&\leq c_7\tau^{1/\beta}\alpha^3.\label{eq:d2}
	\end{align}
	\end{lemma}

\begin{proof}
	The first statement, namely the fact that there exists $c_4>0$ such that~\eqref{eq:c4} holds, follows directly from the formula~\eqref{eq:ftauform} and the fact that $C(q,\alpha,s)$ satisfies estimates analogous to~\eqref{eq:cbound1} and~\eqref{eq:cbound2} when $s$ is small.
	 
	By formula~\eqref{eq:ftauform} and using the bounds~\eqref{eq:cbound1},~\eqref{eq:cbound2} and $h\geq c_4/\alpha$, one obtains
	\begin{align}
			\big|F_\tau(h,\alpha)-F_0(h,\alpha)\big|&\leq\frac{1}{h^{q-1}}\big|C(q,\alpha,\tau^{1/\beta}/h)-C(q,\alpha,0)\big|\notag\\
			&\leq C\frac{\alpha^{q-1}\big[(2\alpha+\tau^{1/\beta}/h)^{q-2}-(2\alpha)^{q-2}\big]}{(2\alpha+\tau^{1/\beta}/h)^{q-2}\alpha^{q-2}}\notag\\
			&\leq c_5\tau^{1/\beta}\alpha.\label{eq:tbound1}
	\end{align}
On the other hand,
\begin{align}
	\partial_h F_\tau(h,\alpha)=\frac{1}{h^2}-\frac{1}{h^q}\Big[(q-1)C(q,\alpha,\tau^{1/\beta}/h)+\frac{\tau^{1/\beta}}{h}\partial_sC(q,\alpha,\tau^{1/\beta}/h)\Big]
\end{align}
where reasoning as in~\eqref{eq:tbound1}
\begin{equation*}
	\frac{1}{h^{q}}\big|C(q,\alpha,\tau^{1/\beta}/h)-C(q,\alpha,0)\big|\leq C\tau^{1/\beta}\alpha^2,
\end{equation*}
while by explicit computation and estimate
\begin{equation*}
	\frac{1}{h^q}|\partial_sC(q,\alpha,\tau^{1/\beta}/h)|\leq C\frac{\alpha^q}{\alpha^{q-1}}\leq C\alpha.
\end{equation*}
Thus we obtain~\eqref{eq:d1}.

The proof of~\eqref{eq:d2} is analogous and follows from an explicit computation of $\partial_h^2F_\tau(h,\alpha)$ and estimates as above, so we omit it.
\end{proof}

We are now ready to prove Theorem~\ref{thm:convex}.

\begin{proof}[Proof of Theorem~\ref{thm:convex}: ]
	Let $0<\eps\ll1$ to be fixed later. By~\eqref{eq:c4} and~\eqref{eq:d0} in Lemma~\ref{lemma:tauest}, there exists $0<\check\tau\leq\tilde\tau$ such that whenever $0\leq\tau\leq\check \tau$ and $F_\tau(h,\alpha)<0$ then $|F_\tau(h,\alpha)-F_0(h,\alpha)|\leq\eps/2$. In particular,
	\begin{equation*}
		F_\tau(h,\alpha)\leq F_0(h^*_{0,\alpha})+\frac{\eps}{2}\quad\Rightarrow\quad F_0(h,\alpha)\leq F_0(h^*_{0,\alpha},\alpha)+\eps.
	\end{equation*}
Thus, by~\eqref{eq:f0min} in Lemma~\ref{lemma:estf0}, one has that $c_1/\alpha\leq h\leq c_2/\alpha$.
Provided $\eps,\check\tau$ are sufficiently small, one can assume that
\[
\frac{c_4}{\alpha}\leq\frac{\bar c_1}{\alpha}\leq\frac{c_1}{\alpha}\leq h\leq\frac{c_2}{\alpha}\leq\frac{\bar c_2}{\alpha}.
\]
Using~\eqref{eq:d2f0} and~\eqref{eq:d2} one obtains
\begin{align}
	|\partial_h^2F_\tau(h,\alpha)|&\geq	|\partial_h^2F_0(h,\alpha)|-	|\partial_h^2F_\tau(h,\alpha)-	\partial_h^2F_0(h,\alpha)|\notag\\
	&\geq c_3\alpha^3-c_7\tau_0^{1/\beta}\alpha^3\notag\\
	&\geq\frac{c_3}{2}\alpha^3>0
\end{align}
In particular, by strict convexity of $F_\tau$ on the region where the minimizers are concentrated, statement $(i)$ is proved. 

Let us now prove $(ii)$. By direct computation,
\begin{equation*}
	\partial_\alpha\Lambda(\tau,\alpha)=\partial_\alpha F_\tau(h^*_{\tau,\alpha},\alpha)=\frac{1}{h^{q-1}}\partial_\alpha C(q,\alpha,\tau^{1/\beta}/h)_{|_{h=h^*_{\tau,\alpha}}}.\label{eq:estderlam2}
\end{equation*}
By estimates analogous to~\eqref{eq:cbound2} one observes that
$|\partial_\alpha C(q,\alpha,\tau^{1/\beta}/h)|\leq C\alpha^{-(q-1)}$ and then 
~\eqref{eq:estderlam} follows from~\eqref{eq:estderlam2} by the fact that $h^*_{\tau,\alpha}\leq\bar c_1/\alpha$.
 
 The last goal is to prove the strict convexity of the function $\alpha\mapsto\Lambda(\tau,\alpha)$. 
 
 By derivation and using the fact that $\partial_h F_\tau(h^*_{h,\alpha},\alpha)=0$, one has that
 \begin{align}\label{eq:derform}
 	\partial^2_\alpha\Lambda(\tau,\alpha)=\Big[\partial^2_\alpha F_\tau(h,\alpha)-\frac{(\partial_\alpha\partial_h F_\tau(h,\alpha))^2}{\partial^2_hF_\tau(h,\alpha)}\Big]_{\Big|_{h=h^*_{\tau,\alpha}}}.
 \end{align}

 Let us first prove the strict convexity of $\alpha\mapsto\Lambda(0,\alpha)$. By~\eqref{eq:derform} and~\eqref{eq:d2f0}, such a function is strictly convex if and only if
 \begin{equation*}
 	\Big[\partial^2_\alpha F_0(h,\alpha)\partial^2_hF_0(h,\alpha)-(\partial_\alpha\partial_hF_0(h,\alpha))^{2}\Big]_{\Big|_{h=h^*_{0,\alpha}}}>0.
 \end{equation*}
By direct computations and using~\eqref{eq:h0form}, one has that
\begin{align}
	\Big[\partial^2_\alpha F_0(h,\alpha)\partial^2_hF_0(h,\alpha)&-(\partial_\alpha\partial_hF_0(h,\alpha))^2\Big]_{\Big|_{h=h^*_{0,\alpha}}}=\notag\\
	&=\frac{(q-1)}{(h^*_{0,\alpha})^{2q}}\big[\partial_\alpha^2C(q,\alpha,0)(q-2)C(q,\alpha,0)-(q-1)^{2}(\partial_\alpha C(q,\alpha,0))^2\big]\notag\\
	&=\frac{8C_1(q-1)(q-2)}{(h^*_{0,\alpha})^{2q}}\Big\{\Big[\sum_{k\geq 0}\frac{1}{(2k+2\alpha)^q}+\frac{1}{(2k+2-2\alpha)^q}\Big]\cdot\notag\\
	&\cdot\Big[\sum_{k\geq 0}\frac{1}{(2k+2\alpha)^{q-2}}+\frac{1}{(2k+2-2\alpha)^{q-2}}-\frac{2}{(2k+2)^{q-2}}\Big]\notag\\
	&-\Big[\sum_{k\geq 0}\frac{1}{(2k+2\alpha)^{q-1}}-\frac{1}{(2k+2-2\alpha)^{q-1}}\Big]^2\Big\}.\label{eq:ineq11}
\end{align}

  For simplicity of notation let us denote by
  \begin{equation*}
	\begin{split}
	  A_{1} :=&\sum_{k\geq 0}\frac{1}{(2k+2\alpha)^q}+\frac{1}{(2k+2-2\alpha)^q} = \frac{1}{(2\alpha)^q} + \sum_{k\geq 1}\frac{1}{(2k+2\alpha)^q}+\frac{1}{(2k-2\alpha)^q}  =\frac{1}{(2\alpha)^q} + \sum_{k\geq1} a_{1,k} \\
	  A_{2} :=&  \sum_{k\geq 0}\frac{1}{(2k+2\alpha)^{q-2}}+\frac{1}{(2k+2-2\alpha)^{q-2}}-\frac{2}{(2k+2)^{q-2}}\\
	  &= \frac{1}{(2\alpha)^{q-2}} + \sum_{k\geq 1}\frac{1}{(2k+2\alpha)^{q-2}}+\frac{1}{(2k-2\alpha)^{q-2}}-\frac{2}{(2k)^{q-2}} = \frac{1}{(2\alpha)^{q-2}} +  \sum_{k\geq1} a_{2,k}\\
	  A_{3} :=&\sum_{k\geq 0}\frac{1}{(2k+2\alpha)^{q-1}}-\frac{1}{(2k+2-2\alpha)^{q-1}}= \frac{1}{(2\alpha)^{q-1}} + \sum_{k\geq1} \frac{1}{(2k +2\alpha)^{q-1}} - \frac{1}{(2k -2\alpha)^{q-1}} \\ &= \frac{1}{(2\alpha)^{q-1}} + \sum_{k\geq1} a_{3,k} \\
	\end{split}
  \end{equation*}

  Note under our hypthesis $\alpha\leq1/2$, it holds $a_{1,k}, a_{2,k} \geq0$ and  $a_{3,k} \leq 0$. Moreover, given that
  $|a_{3,k}| \leq \frac{1}{(2k-1)^{q-1}} - \frac{1}{(2k+1)^{q-1}} $ it is immediate to see that
  \begin{equation*}
	\begin{split}
	  \sum_{k\geq1}|a_{3,k}| \leq \sum_{k\geq1} \frac{1}{(2k-1)^{q}} - \sum_{k\geq1} \frac{1}{(2k+1)^{q}} \leq 1 \leq \frac{1}{(2\alpha)^{q-1}}.
	\end{split}
  \end{equation*}
  Thus  $A_{3}^{2} \leq \frac{1}{(2\alpha)^{2q-2}}$. Hence we have that 
  \begin{equation*}
	\begin{split}
	  A_{1} A_{2} \geq \frac{1}{(2\alpha)^{2q-2}} +\sum_{k\geq1} a_{2,k}\frac{1}{(2\alpha)^{q}} \geq  A^{2}_{3}+\frac{C}{\alpha^{q-2}},
	\end{split}
  \end{equation*}
where we used that $\sum_{k\geq1}a_{2,k}\geq C\alpha^2$ and then
\begin{equation}\label{eq:q-2}
	\partial_\alpha^2C(q,\alpha,0)(q-2)C(q,\alpha,0)-(q-1)^{2}(\partial_\alpha C(q,\alpha,0))^2\geq\frac{C}{(2\alpha)^{q-2}}.
\end{equation}

Let us proceed now to estimate $\alpha\mapsto\partial_\alpha^2\Lambda(\tau,\alpha)$. 

As before, the positivity of $\partial_\alpha^2\Lambda(\tau,\alpha)$ is equivalent to the positivity of $\partial^2_\alpha F_\tau(h,\alpha)\partial^2_hF_\tau(h,\alpha)-(\partial_\alpha\partial_hF_\tau(h,\alpha))^2$ for $h=h^*_{\tau,\alpha}$. 

Computing explicitly the above partial derivatives of $F_\tau$ and using in the computation of $\partial^2_hF_\tau(h^*_{\tau,\alpha},\alpha)$   the fact that $h^*_{\tau,\alpha}$ as critical point of $F_\tau(h,\alpha)$ satisfies
\begin{equation*}
	h^{q-2}=(q-1)C(q,\alpha,\tau^{1/\beta}/h)+\frac{\tau^{1/\beta}}{h}\partial_sC(q,\alpha,\tau^{1/\beta}/h),
\end{equation*}
one has that 
\begin{align}
	\Big[\partial^2_\alpha F_\tau(h,\alpha)&\partial^2_hF_\tau(h,\alpha)-(\partial_\alpha\partial_hF_\tau(h,\alpha))^2\Big]_{\Big|_{h=h^*_{\tau,\alpha}}}\notag\\
	&=\frac{C}{(h^*_{\tau,\alpha})^{2q}}\Big\{\partial^2_\alpha C(q,\alpha,\tau^{1/\beta}/h)(q-2)C(q,\alpha,\tau^{1/\beta})-(q-1)(\partial_\alpha C(q,\alpha,\tau^{1/\beta}/h))^2\Big\}_{\Big|_{h=h^*_{\tau,\alpha}}}\notag\\
	&+\frac{C}{(h^*_{\tau,\alpha})^{2q}}\cdot\frac{\tau^{1/\beta}}{h^*_{\tau,\alpha}}\Big\{\frac{q-3}{q-1}\partial^2_\alpha C(q,\alpha,\tau^{1/\beta}/h)\partial_s C(q,\alpha,\tau^{1/\beta}/h)\notag\\
	&-2\partial_\alpha C(q,\alpha,\tau^{1/\beta}/h)\partial_s\partial_\alpha C(q,\alpha,\tau^{1/\beta}/h)\Big\}_{\Big|_{h=h^*_{\tau,\alpha}}}\notag\\
	&+\frac{C}{(h^*_{\tau,\alpha})^{2q}}\cdot\Big(\frac{\tau^{1/\beta}}{h^*_{\tau,\alpha}}\Big)^2\Big\{\frac{1}{(q-1)^2}(\partial_s\partial_\alpha C(q,\alpha,\tau^{1/\beta}/h))^2\Big\}_{\Big|_{h=h^*_{\tau,\alpha}}}.\label{eq:stima0bis}
\end{align}
For $\tau$ sufficiently small we know that
\begin{equation}\label{eq:hta}
	c_1\tau^{1/\beta}\alpha\leq\frac{\tau^{1/\beta}}{h^*_{\tau,\alpha}}\leq c_2\tau^{1/\beta}\alpha
\end{equation}
for some $c_1,c_2>0$ independent of $\tau$ and $\alpha$.

Hence, reasoning as in the proof of \eqref{eq:q-2}, one has that
\begin{equation}\label{eq:stima1bis}
	\{\partial^2_\alpha C(q,\alpha,\tau^{1/\beta}/h)(q-2)C(q,\alpha,\tau^{1/\beta})-(q-1)(\partial_\alpha C(q,\alpha,\tau^{1/\beta}/h))^2\Big\}_{\Big|_{h=h^*_{\tau,\alpha}}}\geq\frac{C}{(2\alpha+\tau^{1/\beta}/h^*_{\tau,\alpha})^{q-2}}.
\end{equation}
Setting for simplicity of notation
\begin{align}
	B_1(s)&=\frac{1}{(2\alpha+s)^q}+\sum_{k\geq 1}\frac{1}{(2k+2\alpha+s)^q}+\frac{1}{(2k-2\alpha+s)^q}\notag\\
	B_2(s)&=\frac{1}{(2\alpha+s)^{q-1}}+\sum_{k\geq 1}\frac{1}{(2k+2\alpha+s)^{q-1}}-\frac{1}{(2k-2\alpha+s)^{q-1}}-\frac{2}{(2k+s)^{q-1}}\notag\\
	B_3(s)&=\frac{1}{(2\alpha+s)^{q-1}}+\sum_{k\geq 1}\frac{1}{(2k+2\alpha+s)^{q-1}}-\frac{1}{(2k-2\alpha+s)^{q-1}}\notag\\
	B_4(s)&=\frac{1}{(2\alpha+s)^{q}}+\sum_{k\geq 1}\frac{1}{(2k+2\alpha+s)^{q}}-\frac{1}{(2k-2\alpha+s)^{q}}\notag\\
\end{align}
it is fairly easy to see that 
\begin{align}
	\Big\{\frac{q-3}{q-1}&\partial^2_\alpha C(q,\alpha,\tau^{1/\beta}/h)\partial_s C(q,\alpha,\tau^{1/\beta}/h)-2\partial_\alpha C(q,\alpha,\tau^{1/\beta}/h)\partial_s\partial_\alpha C(q,\alpha,\tau^{1/\beta}/h)\Big\}_{\Big|_{h=h^*_{\tau,\alpha}}}=\notag\\
	&=C\Big(-\frac{q-3}{q-1}B_1(s)B_2(s)+2B_3(s)B_4(s)\Big)_{\Big|_{s=\tau^{1/\beta}/h}}.
\end{align}
One then observes that
\begin{align}
\Big(-\frac{q-3}{q-1}B_1(s)B_2(s)+2B_3(s)B_4(s)\Big)_{\Big|_{s=\tau^{1/\beta}/h}}&\geq\Big(-\frac{q-3}{q-1}+2\Big)\frac{C}{(2\alpha+\tau^{1/\beta}/h)^{2q-1}}-\frac{C}{(2\alpha+\tau^{1/\beta}/h)^{q-1}}\notag\\
&\geq -\frac{C}{(2\alpha+\tau^{1/\beta}/h)^{q-1}},	\label{eq:b1b2}
\end{align}
since $-\frac{q-3}{q-1}+2=1$.
Hence, by \eqref{eq:b1b2} and \eqref{eq:hta}, one has that 

\begin{align}\label{eq:stima2bis}
	\frac{\tau^{1/\beta}}{h}	\Big\{\frac{q-3}{q-1}&\partial^2_\alpha C(q,\alpha,\tau^{1/\beta}/h)\partial_s C(q,\alpha,\tau^{1/\beta}/h)-2\partial_\alpha C(q,\alpha,\tau^{1/\beta}/h)\partial_s\partial_\alpha C(q,\alpha,\tau^{1/\beta}/h)\Big\}_{\Big|_{h=h^*_{\tau,\alpha}}}\geq\notag\\
	&\geq -\tau^{1/\beta}\frac{C}{(2\alpha+\tau^{1/\beta}/h)^{q-2}}.
\end{align}
Inserting \eqref{eq:stima1bis} and \eqref{eq:stima2bis} in \eqref{eq:stima0bis} and noticing that the third term of the r.h.s. of \eqref{eq:stima0bis} is nonnegative, one concludes that 
\begin{align}
		\Big[\partial^2_\alpha F_\tau(h,\alpha)&\partial^2_hF_\tau(h,\alpha)-(\partial_\alpha\partial_hF_\tau(h,\alpha))\Big]_{\Big|_{h=h^*_{\tau,\alpha}}}\notag\\
	&\geq\frac{C}{(h^*_{\tau,\alpha})^{2q}}\frac{1}{(2\alpha+\tau^{1/\beta}/h)^{q-2}}(1-\tau^{1/\beta})\label{eq:343}\\
	&\geq \frac{C}{2(h^*_{\tau,\alpha})^{2q}}\frac{1}{(2\alpha+\tau^{1/\beta}/h)^{q-2}}\label{eq:344}
\end{align}
where in passing from~\eqref{eq:343} to~\eqref{eq:344}  we have chosen $\tau\leq\hat{\tau}\leq\check{\tau}$ sufficiently small.

Moreover, using the formula~\eqref{eq:derform} together with~\eqref{eq:344} and~\eqref{eq:d2f0} one gets~\eqref{eq:2derineq}.

\end{proof}

\section{Preliminary lemmas}

In this section we collect a series of Lemmas and Propositions which will be used in the proof of Theorem~\ref{thm:main}. The main novelties in the proofs are contained in the one-dimensional optimization Lemma~\ref{lemma:1D-optimization} and in Lemma~\ref{lemma:stimaLinea}, due to the imposition of a volume constraint.

We start with recalling the following lemma, corresponding to Remark 7.1 in~\cite{dr_arma}.

\begin{lemma}
	\label{rmk:stimax1}
	There exist $\eta_0 > 0$ and  $\tau_{0} > 0$ such that for every  $0<\tau< \tau_0$, whenever  $E\subset \R^d$ and $s^-<s<s^+\in \partial E_{t_{i}^{\perp}}$ are three consecutive points satisfying  $\min(|s - s^- |,|s^+ -s |) <\eta_0$, then $r_{i,\tau}(E,t_{i}^{\perp},s) > 0$. 
	
	In particular, the following estimate holds 
	\begin{equation}
		\label{eq:stimamax1_eq}
		\begin{split}
			r_{i,\tau}(E,t^{\perp}_{i},s) \geq -1 + C_{1}C_{2} \min(|s-s^+ |^{-\beta},\tau^{-1}) + C_{1}C_{2}\min(|s-s^-|^{-\beta} , \tau^{-1})
		\end{split}
	\end{equation}
	where $C_{1}=\int_{\R^{d-1}}\frac{1}{(\|\xi\|_1+1)^p}\d\xi$ and $C_{2}=\frac{1}{(q-1)(q-2)}$. 
	
	Moreover, for any $C>0$ there exist $\eta_C,\tau_C>0$ such that for every $0<\tau\leq\tau_C$ whenever $s^-<s<s^+\in\partial E$ satisfy $\min\{|s-s^-|,\,|s-s^+|\}<\eta_C$ then $r_{i,\tau}(E,t_i^\perp,s)>C$.
	\end{lemma}


It is convenient to introduce  the one-dimensional analogue of~\eqref{eq:ritau}. Given $E\subset \R$  a set of locally finite perimeter and let $s^-, s,s^+\in \partial E$, we define
\begin{equation}\label{eq:rtau1D}
	\begin{split}
		r_{\tau}(E,s) := -1 & + \int_\R |\rho| \widehat{K}_{\tau}(\rho)\d\rho  -  \int_{s^-}^{s} \int_0^{+\infty}  |\chi_{E}(\rho+ u) - \chi_{E}(u)| \widehat{K}_{\tau} (\rho)\d\rho  \du \\ & - \int_{s}^{s^+} \int_{-\infty}^0  |\chi_{E}(\rho+ u) - \chi_{E}(u)| \widehat{K}_{\tau} (\rho)\d\rho  \du. 
	\end{split}
\end{equation}

The quantities defined in~\eqref{eq:ritau} and~\eqref{eq:rtau1D} are related via $r_{i,\tau}(E,t^\perp_i,s) = r_{\tau}(E_{t^\perp_{i}},s)$.

In the next Lemma we recall Lemma 7.5 in~\cite{dr_arma}, containing  a lower bound for the first term of the decomposition~\eqref{eq:decomposition} as $\tau\to0$.

\begin{lemma}
	\label{lemma:technicalBeforeLocalRigidity}
	Let $E_{0}, \{E_{\tau}\}\subset \R$  be a family of sets of locally finite perimeter and $I\subset \R$ be an open bounded interval.   Moreover, assume that $E_{ \tau}\to E_{0}$ in $L^1(I)$.  
	If we denote by $\{k^{0}_{1},\ldots,k^{0}_{m_{0}}\} = \partial E_{0}\cap I $, then
	\begin{equation}
		\label{eq:gstr5}
		\liminf_{\tau\downarrow 0}\sum_{\substack{s\in \partial E_{\tau}\\ s\in I}}r_{\tau}(E_{\tau},s) \geq \sum_{i=1}^{m_{0}-1}(-1 + {C_{1}C_{2}}|k^{0}_{i} - k^{0}_{i+1} |^{-1}),
	\end{equation}
	where $r_{\tau}$ is defined in~\eqref{eq:rtau1D}.
\end{lemma}

The next proposition is at the base of symmetry breaking at scale $l$: on a square of size $l$, if $\tau$ is  sufficiently close to $0$, a bound on the energy corresponds to a bound on the $L^1$-distance to the unions of stripes.  
It corresponds to Lemma 7.6 in~\cite{dr_arma}.

\begin{proposition}[Local Rigidity] 
	\label{lemma:local_rigidity_alpha}
	For every $M > 1,l,\delta > 0$, there exist $\tau_1>0$ and $\bar{\eta} >0$ 
	such that whenever $0<\tau< {\tau}_1$  and $\bar F_{\tau}(E,Q_{l}(z)) < M$ for some $z\in [0,L)^d$ and $E\subset\R^d$ $[0,L)^d$-periodic, with $L>l$, then it holds $D_{\eta}(E,Q_{l}(z))\leq\delta$ for every $\eta < \bar{\eta}$. Moreover $\bar{\eta}$ can be chosen independently  of $\delta$.  Notice that ${\tau}_1$ and $\bar{\eta}$ are independent of $L$.
\end{proposition} 

In particular, one has the following 
\begin{corollary}\label{cor:gammaconv}
	Let $0<\tau\leq\tau_0\ll1$. One has that the following holds:
	\begin{itemize}
		\item Let $\{E_{\tau}\}$ be a sequence such that $\sup_{\tau} \bar{F}_{ \tau}(E_{\tau}, Q_l(z)) < \infty$. 
		Then the sets $E_{\tau}$ converge in $L^1$  up to  subsequences to some set $E_{0}$ of finite perimeter and 
		\begin{align}
			\liminf_{\tau \rightarrow 0} \bar{F}_{ \tau}(E_{\tau}, Q_l(z)) \geq \bar{F}_{0}(E_{0}, Q_l(z)) . 
			\label{eq:liminfLocalGamma}
		\end{align}
		\item For every set $E_{0}$ with $\bar{F}_{0}(E_{0}, Q_l(z)) < + \infty$, there exists a sequence $\{E_{\tau}\}$ converging in $L^1$ to $E_{0}$ and such that
		\begin{align}
			\limsup_{\tau \rightarrow 0} \bar{F}_{ \tau}(E_{\tau}, Q_l(z)) = \bar{F}_{0}(E_{0}, Q_l(z)) . 
		\end{align} 
	\end{itemize}
\end{corollary}

The following proposition corresponds to Lemma 7.8 in~\cite{dr_arma}. Roughly speaking,  it shows that if  we are in a cube where the set $E\subset\R^d$ is close to a set $E'$ which is a union of stripes in direction $e_i$ (according to Definition~\ref{def:defDEta}), then it is not convenient to oscillate in direction $e_j$ with $j\neq i$ (namely, on the slices in direction $e_i$ to have points in $\partial E_{t_i^\perp}$).
Indeed, in such a case either the local contribution given by $r_{i,\tau}$ or the one given by $v_{i,\tau}$ are large.

\begin{proposition}[Local Stability]
	\label{lemma:stimaContributoVariazionePiccola}
	Let  $(t^{\perp}_{i}+se_i)\in (\partial E) \cap [0,l)^d$,  and  $\eta_{0}$, $\tau_0$ as in Lemma~\ref{rmk:stimax1}. Then there exist ${\tau_2},\varepsilon_2$ (independent of $l$) such that for every $0<\tau < {\tau_2}$, and $0<\varepsilon < {\varepsilon_2}$ the following holds: assume that 
	\begin{enumerate}[(a)]
		\item $\min(|s-l|, |s|)> \eta_0$ (i.e. the boundary point $s$ in the slice of $E$ is sufficiently far from the boundary of the cube)
		\item $D^{j}_{\eta}(E,[0,l)^d)\leq\frac {\varepsilon^d} {16 l^d}$ for some $\eta> 0$ and  with $j\neq i$ (i.e. $E\cap [0,l)^d$ is close to stripes with boundaries orthogonal to $e_j$ for some $j\neq i$)
	\end{enumerate}
	Then 
	\[r_{i,\tau}(E,t^{\perp}_{i},s) + v_{i,\tau}(E,t^{\perp}_{i},s) \geq 0.\] 
\end{proposition}

The following lemma contains  the main one-dimensional estimate needed in the proof  of Lemma~\ref{lemma:stimaLinea}, which corresponds to a one-dimensional optimization taking into account boundary effects and the density of the set on a given interval.

In the proof we use a periodic extension argument with volume constraint (which gives the error term $C_0(\eta_C)$ for some fixed $C>0$, $\eta_C$ as in Lemma~\ref{rmk:stimax1}) and then the fact that for periodic one-dimensional sets with density $\alpha$ (i.e. sets in $\mathcal C_{L,\alpha}$) the energy contribution is bigger than or equal to the contribution of periodic stripes of density $\alpha$ and period  $h^*_{\tau,\alpha}$ (i.e. sets of the form $E_{h^*_{\tau,\alpha},\alpha}$). The convexity in $\alpha$ of the function $\Lambda(\tau,\alpha)$ and the boundedness of its derivative proven in Theorem \ref{thm:convex} will be also crucial.
\begin{lemma}
	\label{lemma:1D-optimization}
	There exists $C>0$ and $C_0=C_0(\eta_C)$ with $\eta_C=\eta_C(C)$ as in Lemma~\ref{rmk:stimax1} such that the following holds.
	Let $E\subset \R$  be a set of locally finite perimeter and $I\subset \R$ be an open interval. 
	Let $r_{\tau}(E,s)$ be defined as in~\eqref{eq:rtau1D}.
	Then for all $0<\tau< \min\{\hat \tau,\tau_C\}$, where $\hat \tau$ is given in Theorem~\ref{thm:convex} and  $\tau_C$ is given in Remark~\ref{rmk:stimax1}, it holds
	\begin{equation}
		\label{eq:gstr40}
		\sum_{\substack{s \in \partial E\\ s \in I}} r_{\tau}(E,s) \geq |I|\Lambda(\tau,\alpha(I)) - C_0,
	\end{equation}
where 
\begin{equation}\label{eq:alphadiI}
	\alpha(I)=\frac{|E\cap I|}{|I|}.
\end{equation}
\end{lemma}

\begin{proof}
	Let us denote by $k_1< \ldots< k_m $ the  points of $\partial E \cap I$, and 
	
	\begin{equation*}
		k_0 = \sup\{ s\in \partial E: s < k_1\} \qquad\text{and}\qquad 
		k_{m+1} = \inf\{ s\in \partial E: s > k_m\} 
	\end{equation*}
	
	Without loss of generality we may assume that  $r_{\tau}(E,k_1) < C$ and that $r_{\tau}(E,k_m) < C$ for some $C>0$ to be fixed later.  
	If that is not the case we can restrict $I$ to a smaller interval $I'=[k_j,k_\ell]$ where $k_j,k_\ell\in\partial E$ and  $\min\{|k_{j+1}-k_j|,|k_j-k_{j-1}|,|k_{\ell-1}-k_\ell|,|k_\ell-k_{\ell+1}|\}\geq\eta_C$, with $\eta_C$ as in Lemma~\ref{rmk:stimax1}. Setting $b=m-(\ell-j+1)$ to be the number of removed boundary points, and assuming to be able to prove our result for $I'$, we  obtain that
	\begin{equation}\label{eq:17.5}
		\sum_{\substack{s \in \partial E\\ s \in I}} r_{\tau}(E,s) \geq Cb+\sum_{\substack{s \in \partial E\\ s \in I'}} r_{\tau}(E,s) \geq Cb+ \Lambda(\tau,\alpha(I')) |I'| - C_0. 
	\end{equation}
Using the convexity of $\alpha\mapsto\Lambda(\tau,\alpha)$ and the negativity of $\Lambda(\tau,\alpha)$ for $\tau\leq\hat\tau$ as in Theorem~\ref{thm:convex},  together with $|I|\geq|I'|$, one has that
\begin{align}
\Lambda(\tau,\alpha(I')) |I'|&\geq\Lambda(\tau,\alpha(I)) |I'|+(\alpha(I')-\alpha(I))\partial_a\Lambda(\tau,\alpha(I))|I'|\notag\\
&\geq\Lambda(\tau,\alpha(I)) |I|+(\alpha(I')-\alpha(I))\partial_a\Lambda(\tau,\alpha(I))|I'|.\label{eq:27.5}
\end{align}
Now observe that by point (ii) in Theorem~\ref{thm:convex} one has that for $0\leq\tau\leq\hat\tau$ it holds  $|\partial_\alpha\Lambda(\tau,\alpha)|\leq\bar C$.

 Moreover, setting $a=|I'|$ and $(b+\gamma)\eta_C=|I\setminus I'|$ for some $\gamma\in(-1,1)$, 
\begin{align}
(\alpha(I')-\alpha(I))|I'|&=\frac{(a+(b+\gamma)\eta_C)|E\cap I'|-a|E\cap I'|-a|E\cap(I\setminus I')|}{(a+(b+\gamma)\eta_C)}\notag\\
&=a\frac{(b+\gamma)\eta_C\alpha(I')-|E\cap(I\setminus I')|}{(a+(b+\gamma)\eta_C)}.	
\end{align}
Hence 
\begin{equation*}
	|\alpha(I')-\alpha(I)||I'|\leq\frac{a}{a+(b+\gamma)\eta_C}(b+\gamma)\eta_C\max\{\alpha(I'), \alpha(I\setminus I')\}\leq (b+1)\eta_C.\label{eq:etac}
\end{equation*}
Therefore, provided $Cb>\bar C(b+1)\eta_C$, which can be achieved by taking $C$ sufficiently large and then $\eta_C$ sufficiently small, one has that $Cb+(\alpha(I')-\alpha(I))\partial_\alpha\Lambda(\tau, \alpha(I))|I'|\geq0$ and then by~\eqref{eq:17.5},~\eqref{eq:27.5} and the above
\begin{equation*}
		\sum_{\substack{s \in \partial E\\ s \in I}} r_{\tau}(E,s)\geq \Lambda(\tau,\alpha(I)) |I|-C_0.
\end{equation*}
 
Then, we know that we can reduce to the case in which $I=[k_1,k_m]$ and 	
	\begin{equation*}
		\min(|k_1 - k_0 |, |k_2 - k_1 |, |k_{m-1} - k_m |,|k_{m+1} - k_m |) > \eta_C. 
	\end{equation*}
	
	We claim that 
	\begin{equation}
		\label{eq:gstr4}
		\sum_{i =1}^{ m } r_{\tau}(E,k_i) \geq \sum_{i =1}^{ m } r_{\tau}(E',k_i)  - \bar C_0 
	\end{equation}
	where $E'$ is obtained by extending periodically $E$ with the pattern contained in $E \cap (k_1,k_m)$ and $ \bar C_0  = \bar 	C_0(\eta_C) > 0$. 
	The construction of $E'$ can be done as follows: if $m$ is odd we repeat periodically $E\cap (k_1,k_m)$, and if $m$ is even we repeat periodically $E\cap\bar I=E\cap(k_1-\eta_{C}/2, k_m)$. 
	
	Thus we have constructed a set $E'$ which is periodic of period $|I|=k_m-k_1$ or $|\bar I|=k_m- k_1 + \eta_C/2$. 
	Therefore
	\begin{align}
		\sum_{i=1}^m r_\tau(E',k_i) &\geq \Lambda\Big(\tau,\frac{|E'\cap\bar I|}{|\bar I|}\Big)|\bar I|\notag\\
		&\geq \Lambda(\tau, \alpha (I))(|I|+\eta_C/2)+\partial_\alpha\Lambda(\tau,\alpha(I))\Big(\frac{|E'\cap\bar I|}{|\bar I|}-\alpha(I)\Big)(|I|+\eta_C/2).
				\label{eq:ggstr1}
	\end{align}
	Indeed, in order to obtain the first inequality above it is sufficient to notice that, by reflection positivity Theorem~\ref{thm:giu}, periodic stripes with width $h^*_{\tau,\alpha}$ and density $\alpha$ are those with minimal energy density among all periodic sets with volume density $\alpha=\frac{|E'\cap\bar I|}{|\bar I|}$,  and $\Lambda(\tau,\alpha)$ was defined as the minimal energy density in this class.
	
As in~\eqref{eq:etac}, one can see that
\[
\Big|\frac{|E'\cap\bar I|}{|\bar I|}-\alpha(I)\Big|\lesssim\eta_C,
\] 
hence from~\eqref{eq:ggstr1} and the bound~\eqref{eq:estderlam} it follows that
\begin{align}
		\sum_{i=1}^m r_\tau(E',k_i) &\geq\Lambda(\tau,\alpha(I))|I|-\tilde C_0,\label{eq:ggstrlast}
\end{align}
with $\tilde C_0=\tilde C_0(\eta_C)$.

	Inequality~\ref{eq:ggstrlast} combined with~\eqref{eq:gstr4} yields~\eqref{eq:gstr40}.
	
	To show~\eqref{eq:gstr4}, the difference between $E$ and $E'$ satisfies
	\begin{equation*}
		E \Delta E'  \subset (- \infty, k_1 - \eta_C/2) \cup (k_m+\eta_C, + \infty),
	\end{equation*}
	where $\eta_{C}$ is the constant defined in Remark~\ref{rmk:stimax1}.
	To obtain~\eqref{eq:gstr4},  we need to estimate $|\sum_{i=1}^m r_\tau(E,k_i) - \sum_{i=1}^m r_\tau(E',k_i) |$.
	Let 
	\begin{equation*}
		\begin{split}
			\sum_{i=1}^m r_\tau(E,k_i) - \sum_{i=1}^m r_\tau(E',k_i) =
			A + B,
		\end{split}
	\end{equation*}
	where
	\begin{equation*}
		\begin{split}
			A =  \sum_{i=0}^{m-1}&  \int_{k_{i}}^{k_{i+1}}\int_{0}^{+\infty} (s - |\chi_{E}(s+u) - \chi_{E}(u)|)\hat{K}_{\tau}(s) \ds \du
			\\  & - \sum_{i=0}^{m-1} \int_{k_{i}}^{k_{i+1}}\int_{0}^{+\infty} (s - |\chi_{E'}(s+u) - \chi_{E'}(u)|)\hat{K}_\tau(s) \ds \du\\ 
			B =  \sum_{i=1}^{m} & \int_{k_{i}}^{k_{i+1}}\int_{-\infty}^{0} (s - |\chi_{E}(s+u) - \chi_{E}(u)|)\hat{K}_\tau(s) \ds \du
			\\ &- \sum_{i=1}^{m} \int_{k_{i}}^{k_{i+1}}\int_{-\infty}^{0} (s - |\chi_{E'}(s+u) - \chi_{E'}(u)|)\hat{K}_\tau (s)\ds \du. 
		\end{split}
	\end{equation*}
	Thus we have that 
	\begin{equation*}
		\begin{split}
			|A| \leq \int_{k_{0}} ^{k_{m}} \int_{0}^{+\infty} \chi_{E \Delta E'}(u +s  ) \hat{K}_{\tau}(s) \ds \du \leq\int_{k_0}^{k_m}\int_{k_{m} + \eta_C}^{\infty} \hat{K}_{\tau}(u-v) \dv\du \leq  \frac{\bar C_{0}}{2},
		\end{split}
	\end{equation*}
	where $\bar C_{0}$ is a constant depending only on $\eta_C$.  Similarly, $|B| \leq \bar C_{0}/2$

	Thus we have that
	\begin{align*}
		\Big|\sum_{i=1}^m r_\tau(E,k_i) - \sum_{i=1}^m r_\tau(E',k_i)\Big| &\leq \int_{k_0}^{k_m}\int_{k_{m} + \eta_C}^{\infty} \hat{K}_{\tau}(u-v) \du\dv
		+ \int_{k_1}^{k_{m+1}}\int_{-\infty}^{k_{1} - \eta_C/2} \hat{K}_{\tau}(u-v) \dv\du\\
		&\leq \bar C_0.
	\end{align*}
	
\end{proof}


The next lemma is the analogue of Lemma 7.11 in~\cite{dr_arma} and gives a lower bound on the energy in the case almost all the volume of $Q_l(z)$ is filled by $E$ or $E^c$ (this will be the case on the set $A_{-1}$ defined in~\eqref{a1}).
\begin{lemma}
	\label{lemma:stimaQuasiPieno}
	Let  $E$ be a set of locally finite perimeter  such that $\min(|Q_{l}(z)\setminus E|, |E\cap Q_{l}(z) |)\leq {\delta} l^d$, for some $\delta>0$. Then 
	\begin{equation*}
		\begin{split}
			\bar F_{\tau} (E,Q_{l}(z)) \geq -\frac {\delta d } {\eta_0 },
		\end{split}
	\end{equation*}
	where $\eta_0$ is defined in Lemma~\ref{rmk:stimax1}.
\end{lemma}

The following lemma contains the main lower bounds of the functional along one-dimensional slices and uses the results of the previous lemmas of this section.  If compared with the estimates along slices used for the minimization problem without volume constraint in \cite{dr_arma}, the effect of the boundary in~\eqref{eq:gstr27} and~\eqref{eq:gstr36} together with the density on a $d$-dimensional neighbourhood of the slice  have now to be taken into account. In particular, this leads to consider the intervals $J_l$ in \eqref{eq:jl} below and a different partition into sets $A_{i,l}$ and $B_l$ in~\eqref{eq:albl}.

\begin{lemma}
	\label{lemma:stimaLinea}
	Let  ${\eps_2},{\tau_2}>0$ as in Lemma~\ref{lemma:stimaContributoVariazionePiccola}. Let $\delta=\eps^d/(16l^d)$ with $0<\eps\leq{\eps_2}$, $0<\tau\leq{\tau_2}$ and 
	$C_0$ be the constant appearing in Lemma~\ref{lemma:1D-optimization}. Let $t_i^\perp\in[0,L)^{d-1}$ and $\eta>0$.

	The following hold: there exists a constant $C_1$ independent of $l$ (but depending on the dimension and on $\eta_0$ as in Lemma~\ref{rmk:stimax1}) such that
	\begin{enumerate}[(i)]
		\item Let $J\subset \R$ an interval  such that for every $s\in J$ one has that  $D^{j}_{\eta}(E,Q_{l}(t^{\perp}_{i}+se_i))\leq \delta$ with $j\neq i$. 
		Then
		\begin{equation}
			\label{eq:gstr20}
			\begin{split}
				\int_{J} \bar{F}_{i,\alpha,\tau}(E,Q_{l}(t^{\perp}_{i}+se_i))\ds \geq - \frac{C_1}{l}.
			\end{split}
		\end{equation}
		Moreover, if $J = [0,L)$, then 
		\begin{equation}
			\label{eq:gstr21}
			\begin{split}
				\int_{J} \bar{F}_{i,\alpha,\tau}(E,Q_{l}(t^{\perp}_{i}+se_i))\ds \geq0.
			\end{split}
		\end{equation}
		\item Let $J = (a,b)\subset \R$. 
		If for $s=a$ and $s=b$ it holds $D_\eta^j(E,Q_{l}(t^{\perp}_i+se_i)) \leq \delta$ with $j\neq i$, then  setting 
		\begin{equation}\label{eq:jl}
			J_l=(a+l/2,b-l/2)\text{ if $|b-a|>l$, } J_l=\emptyset \text{ otherwise}
		\end{equation} and, as in~\eqref{eq:alphadiI},
		\[
		\alpha(J_le_i+Q_l^\perp(t_i^\perp))=\frac{|E\cap (J_le_i+Q_l^\perp(t_i^\perp))|}{|J_l|l^{d-1}},
		\]
		one has that
		\begin{equation}
			\label{eq:gstr27}
			\begin{split}
				\int_{J} \bar{F}_{i,\alpha,\tau}(E,Q_{l}(t^{\perp}_{i}+se_i))\ds \geq\Big(|J_l| \Lambda(\tau,\alpha(J_le_i+Q_l^\perp(t_i^\perp))) -C_0\Big)\chi_{(0,+\infty)}(|J|-l)-\frac{C_1} l,
			\end{split}
		\end{equation}
		otherwise
		\begin{equation}
			\label{eq:gstr36}
			\begin{split}
				\int_{J} \bar{F}_{i,\alpha,\tau}(E,Q_{l}(t^{\perp}_{i}+se_i))\ds \geq |J_l|\Lambda(\tau,\alpha(J_le_i+Q_l^\perp(t_i^\perp)))\chi_{(0,+\infty)}(|J|-l) - C_1l.
			\end{split}
		\end{equation}
		Moreover, if $J = [0,L)$, then
		\begin{equation}
			\label{eq:gstr28}
			\begin{split}
				\int_{J} \bar{F}_{i,\alpha,\tau}(E,Q_{l}(t^{\perp}_{i}+se_i))\ds \geq |J|\Lambda(\tau,\alpha(Je_i+Q_l^\perp (t_i^\perp))).
			\end{split}
		\end{equation}
	\end{enumerate}
\end{lemma}

   \begin{proof}The proof of $(i)$ follows from Lemma~\ref{lemma:stimaContributoVariazionePiccola} as in Lemma 7.9 in~\cite{dr_arma}.

Let us now prove $(ii)$.  Without loss of generality let us assume that $J= (0,l')$.  

One has that

   \begin{equation}
      \label{eq:gstr32}
      \begin{split}
         \int_{J} \bar{F}_{i,\tau}(E,Q_{l}(t^{\perp}_i,s)) \ds &         \geq \frac{1}{l^{d-1}} \int_{Q^{\perp}_{l}(t^{\perp}_{i})} \sum_{\substack{s' \in \partial E_{t'^{\perp}_{i}} 
               \\ s'\in (-\frac l2, l' + \frac l2)}}\frac{|Q^{i}_{l}(s')\cap J|}{l} \Big(r_{i,\tau }(E,t'^{\perp}_{i},s') + v_{i,\tau }(E,t'^{\perp}_{i},s')\Big) \dt'^{\perp}_{i}\\
         & =   \frac{1}{l^{d-1}} \int_{Q^{\perp}_{l}(t^{\perp}_{i})} \sum_{\substack{s' \in \partial E_{t'^{\perp}_{i}} 
               \\ s'\in (\frac l2, l'-\frac l2 )}}\frac{|Q^{i}_{l}(s')\cap J|}{l} \Big(r_{i,\tau }(E,t'^{\perp}_{i},s') + v_{i,\tau }(E,t'^{\perp}_{i},s')\Big) \dt'^{\perp}_{i}\\
         & +   \frac{1}{l^{d-1}} \int_{Q^{\perp}_{l}(t^{\perp}_{i})} \sum_{\substack{s' \in \partial E_{t'^{\perp}_{i}} 
               \\ s'\in (-\frac l2, \frac l2 ]\cup [l'-\frac l2,l'+\frac l2)}}\frac{|Q^{i}_{l}(s')\cap J|}{l} \Big(r_{i,\tau }(E,t'^{\perp}_{i},s') + v_{i,\tau }(E,t'^{\perp}_{i},s')\Big) \dt'^{\perp}_{i}
      \end{split}
   \end{equation}
   where if $l'\leq l$, we have that $(l/2,l'-l/2)$ is empty and then 
   \begin{equation*}
   \begin{split}
   \frac{1}{l^{d-1}} \int_{Q^{\perp}_{l}(t^{\perp}_{i})} \sum_{\substack{s' \in \partial E_{t'^{\perp}_{i}} 
   		\\ s'\in (l/2, l'-l/2)}} \Big(r_{i,\tau }(E,t'^{\perp}_{i},s')  + v_{i,\tau }(E,t'^{\perp}_{i},s')\Big) \dt'^{\perp}_{i} = 0.
   \end{split}
   \end{equation*}

   Fix $t'^\perp_i$. 
   We will now estimate the contributions for $(t_i'^\perp,s')\in Q_{l}(t'^\perp_i,0)$ 
   and $(t_i'^\perp,s')\in Q_{l}(t'^\perp_i,l')$. 
   If the condition $D^{j}_{\eta}(E,Q_{l}(t'^\perp_i,0)) \leq \delta$ or $D^{j}_{\eta}(E,Q_{l}(t'^\perp_i,l')) \leq \delta$ is missing, then we will estimate $r_{i,\tau}(E,t'^\perp_{i},s') + v_{i,\tau}(E,t'^{\perp}_{i},s')$ from below with $-1$ whenever the neighbouring ``jump'' points are further than $\eta_0$, otherwise $r_{i,\tau} \geq 0$. 
   Hence, the last term in~\eqref{eq:gstr32} can be estimated by
   \begin{equation*}
      \begin{split}
         &    \frac{1}{l^{d-1}} \int_{Q^{\perp}_{l}(t^{\perp}_{i})} \sum_{\substack{s' \in \partial E_{t'^{\perp}_{i}} 
               \\ s'\in (-\frac l2, \frac l2 ]}}\frac{|Q^{i}_{l}(s')\cap J|}{l} \Big(r_{i,\tau }(E,t'^{\perp}_{i},s') + v_{i,\tau }(E,t'^{\perp}_{i},s')\Big) \dt'^{\perp}_{i} \gtrsim - M_{0}l
         \\ 
         &\frac{1}{l^{d-1}} \int_{Q^{\perp}_{l}(t^{\perp}_{i})} \sum_{\substack{s' \in \partial E_{t'^{\perp}_{i}} 
               \\ s'\in [l' -\frac l2,l'+  \frac l2 )}}\frac{|Q^{i}_{l}(s')\cap J|}{l} \Big(r_{i,\tau }(E,t'^{\perp}_{i},s') + v_{i,\tau }(E,t'^{\perp}_{i},s')\Big) \dt'^{\perp}_{i} \gtrsim - M_{0}l. 
      \end{split}
   \end{equation*}
   If $l' > l$,  we have that, by convexity of $\alpha\mapsto\Lambda(\tau,\alpha)$ proven in Theorem \ref{thm:convex} for $\tau\leq\hat{\tau}$,  
   \begin{align*}
   \frac{1}{l^{d-1}} \int_{Q^{\perp}_{l}(t^{\perp}_{i})} \sum_{\substack{s' \in \partial E_{t'^{\perp}_{i}} 
   		\\ s'\in (l/2, l'-l/2)}} r_{i,\tau }(E, t'^{\perp}_{i},s')  \dt'^{\perp}_{i} &\geq    \frac{1}{l^{d-1}} \int_{Q^{\perp}_{l}(t^{\perp}_{i})} \Lambda(\tau,\alpha(J_le_i+t'^\perp_{i})) |J_l|\dt_i'^\perp-{C_0}\\
   	&\geq |J_l|\Lambda(\tau,\alpha(J_le_i+Q_l^\perp(t^\perp_{i})))-{C_0}
   \end{align*}
   where in the last inequality we have used Lemma~\ref{lemma:1D-optimization} for $E=E_{t'^{\perp}_{i}}$ and $J_l=(l/2,l'-l/2)$. 
 
   Combining the above with the fact that the same term is $0$ if $l'<l$, we have that
   \begin{equation*}
   \begin{split}
   \frac{1}{l^{d-1}} \int_{Q^{\perp}_{l}(t^{\perp}_{i})} \sum_{\substack{s' \in \partial E_{t'^{\perp}_{i}} 
   		\\ s'\in (l/2, l'-l/2)}} r_{i,\tau }(E,t'^{\perp}_{i},s')  \dt'^{\perp}_{i} \geq \big(|J_l|\Lambda(\tau,\alpha(J_le_i+Q_l^\perp(t_{i}^\perp)))-{C_0}\big) \chi_{(0, +\infty)}(|J| -l),
   \end{split}
   \end{equation*}
   
   Thus ~\eqref{eq:gstr36} follows.

   Let us now turn to the proof of~\eqref{eq:gstr27}.
 Given that $D^j_{\eta}(E,Q_{l}(t^{\perp}_{i},0))\leq \delta$ and $D^{j}_{\eta}(E,Q_{l}(t^\perp_i,l')) \leq\delta$ for some $j\neq i$, by Lemma~\ref{lemma:stimaContributoVariazionePiccola} with $\delta=\varepsilon^d/(16l^d)$ we have that
   \begin{equation}
      \label{eq:gstr37}
      \begin{split}
         r_{i,\tau}(E,t'^\perp_i,s') + v_{i,\tau}(E,t'^\perp_i,s') \geq  0 
      \end{split}
   \end{equation}
   whenever $\min(|s' - l'+l/2| ,|s'- l' -l/2 |) \geq \eta_0$  and $(t'^\perp_i,s')\in Q_{l}(t^\perp_i,l')$
   or  $\min(|s'+l/2| ,|s'-l/2 |)\geq \eta_0$  and $(t'^\perp_i,s')\in Q_{l}(t^\perp_i,0)$. 

   Fix $t'^{\perp}_{i}$. Then
   \begin{equation*}
      \begin{split}
          \sum_{\substack{s' \in \partial E_{t'^{\perp}_{i}} 
               \\ s'\in (-\frac l2, \frac l2 )}}\frac{|Q^{i}_{l}(s')\cap J|}{l} & \Big(r_{i,\tau }(E,t'^{\perp}_{i},s') + v_{i,\tau }(E,t'^{\perp}_{i},s')\Big) 
         \\  \geq & \sum_{\substack{s' \in \partial E_{t'^{\perp}_{i}} 
               \\ s'\in (-\frac l2, \frac l2 )\\
            \min(|s'+l/2| ,|s'-l/2 |)\geq \eta_0
            }}\frac{|Q^{i}_{l}(s')\cap J|}{l} \Big(r_{i,\tau }(E,t'^{\perp}_{i},s') + v_{i,\tau }(E,t'^{\perp}_{i},s')\Big)
         \\  + & \sum_{\substack{s' \in \partial E_{t'^{\perp}_{i}} 
               \\ s'\in (-\frac l2, \frac l2 )\\
            \min(|s'+l/2| ,|s'-l/2 |)<  \eta_0
            }}\frac{|Q^{i}_{l}(s')\cap J|}{l} \Big(r_{i,\tau }(E,t'^{\perp}_{i},s') + v_{i,\tau }(E,t'^{\perp}_{i},s')\Big) 
      \end{split}
   \end{equation*}
   Thus by using~\eqref{eq:gstr37}, we have that the first term on the \rhs above is positive. To estimate the last term on the \rhs above we notice that $r_{i,\tau} \geq 0 $ whenever the neighbouring points are closer than $\eta_0$ and otherwise $r_{i,\tau}\geq  -1$.
   Moreover, given that  $\frac{|Q^i_l{(s')}\cap J |}{l} < \frac{\eta_0}{l}$ for $s' \in (-l/2,l/2)\cup (l'-l/2,l'+l/2)$, we have that the last term on the \rhs above can be bounded from below by $-M_0/l$. Finally integrating over $t'^\perp_i$ we obtain that
   \begin{equation*}
      \begin{split}
         \frac{1}{l^{d-1}} \int_{Q^{\perp}_{l}(t^{\perp}_{i})} \sum_{\substack{s' \in \partial E_{t'^{\perp}_{i}} 
               \\ s'\in (-l/2, l/2 )\cup (l'-l/2,l'+l/2)}}\frac{|Q^{i}_{l}(s')\cap J|}{l} \Big(r_{i,\tau }(E,t'^{\perp}_{i},s') + v_{i,\tau }(E,t'^{\perp}_{i},s')\Big) \dt'^{\perp}_{i} \gtrsim -\frac{M_{0}}{l}.
      \end{split}
   \end{equation*}

   By using the above inequality in~\eqref{eq:gstr32} and  the fact  that for every $s'\in (l/2,l'-l/2)$ it holds $\frac{|Q_{l}(s')\cap J |}{l} =1 $, we have that
   \begin{equation*}
      \begin{split}
         \int_{J} \bar{F}_{i,\tau}(E,Q_{l}(t^{\perp}_i,s)) \ds 
         & \geq   \frac{1}{l^{d-1}} \int_{Q^{\perp}_{l}(t^{\perp}_{i})} \sum_{\substack{s' \in \partial E_{t'^{\perp}_{i}} 
               \\ s'\in (l/2, l'-l/2 )}} \Big(r_{i,\tau }(E,t'^{\perp}_{i},s') + v_{i,\tau }(E,t'^{\perp}_{i},s')\Big) \dt'^{\perp}_{i}
         - \frac{M_{0}}{l}
      \end{split}
   \end{equation*}

   To conclude the proof of~\eqref{eq:gstr27}, as for~\eqref{eq:gstr36}, we notice that
   \begin{equation*}
      \begin{split}
         \frac{1}{l^{d-1}} \int_{Q^{\perp}_{l}(t^{\perp}_{i})} \sum_{\substack{s' \in \partial E_{t'^{\perp}_{i}} 
               \\ s'\in (l/2, l'-l/2)}} r_{i,\tau }(E,t'^{\perp}_{i},s')  \dt'^{\perp}_{i} \geq \big(|J_l|\Lambda(\tau,\alpha(J_le_i+Q_l^\perp(t_{i}^\perp)))-{C_0}\big) \chi_{(0, +\infty)}(|J| -l),
      \end{split}
   \end{equation*}
   where in the last inequality we have used Lemma~\ref{lemma:1D-optimization} for $E=E_{t'^{\perp}_{i}}$, $J_l=(l/2,l'-l/2)$. Hence one gets~\eqref{eq:gstr27}.

 If $|J|\leq l$, then the first sum on the \rhs of~\eqref{eq:gstr32} is performed on an empty set. Therefore, in both~\eqref{eq:gstr27} and~\eqref{eq:gstr36} one has only the boundary terms and can conclude in a similar way.

The proof of~\eqref{eq:gstr28} proceeds using the $L$-periodicity of the contributions.

\end{proof}

\section{Proof of Theorem~\ref{thm:main}}
\subsection{Setting the parameters}\label{Ss:param}
The sets defined in the proof and the main estimates will depend on a set of parameters $l,\delta,\rho,M, \eta$ and $\tau$.  Our aim now is to fix such parameters, making explicit their dependence on each other. We will refer to such choices during the proof of the main theorem. Due to the symmetry of the problem w.r.t. $\alpha=1/2$, in the following we consider for simplicity densities $\alpha\in(0,1/2]$.

\begin{enumerate}
	\item We first fix $\eta_0,\tau_0$ as in Lemma~\ref{rmk:stimax1} and $\hat \tau$ as in Theorem~\ref{thm:convex}.
	
	\item We fix then $\bar \alpha\in(0,1/2]$ and choose  a global volume constraint $\alpha\in[\bar \alpha,1/2]$ on $[0,L)^d$. We let $C$ be as in Lemma~\ref{lemma:1D-optimization} and $\tau_C$ as in Lemma~\ref{rmk:stimax1}.
	
	\item  Let then  $l>0$ s.t.
	\begin{equation}\label{eq:lfix}
		l>\frac{ c_q(2C_0+ 2dC(d,\eta_0))}{\bar{\alpha}^{q+2}}, 
	\end{equation}
	where $C(d,\eta_0)$ is the constant (depending only on the dimension $d$ and on $\eta_0$) defined in~\eqref{eq:ca0},
	$C_0=C_0(\eta_C)$ is the constant which appears in the statement of Lemma~\ref{lemma:1D-optimization} and $c_q$ is the constant appearing in~\eqref{eq:lfinal}

	\item We  find  the parameters  ${\varepsilon}_2 = {\varepsilon}_2(\eta_0,\tau_0)$ and ${\tau}_2 = {\tau}_2(\eta_0, \tau_0)$ as in Proposition~\ref{lemma:stimaContributoVariazionePiccola}.
	
	\item We consider then  $\varepsilon \leq {\varepsilon}_2$, $\tau \leq {\tau}_2$ as in Lemma~\ref{lemma:stimaLinea}. We define   $\delta$ as $\delta  = \frac{\varepsilon^d}{16}$. Moreover,  by choosing $\varepsilon$ sufficiently small we can additionally assume that
	\begin{equation}\label{eq:deltafix2}
		D^i_{\eta}(E,Q_l(z))\leq\delta\text{ and }D^j_\eta(E,Q_l(z))\leq\delta,\:i\neq j\quad\Rightarrow\quad\min\{|E\cap Q_l(z)|,|E^c\cap Q_l(z)|\}\leq l^{d-1}.  
	\end{equation}
	The above  follows from Remark~\ref{rmk:lip} (ii). 
	
	\item By Remark~\ref{rmk:lip} (i), we then fix
	\begin{equation}\label{eq:rhofix}
		\rho\sim\delta l. 
	\end{equation}
	in such a way that  for any $\eta$ the following holds
	\begin{equation}\label{eq:rhofix2}
		\forall\,z,z'\text{ s.t. }D_{\eta}(E,Q_l(z))\geq\delta,\:|z-z'|_\infty\leq\rho\quad\Rightarrow\quad D_\eta(E,Q_l(z'))\geq\delta/2.
	\end{equation}

	\item Then we fix $M$ such that
	\begin{equation}
		\label{eq:Mfix}
		\frac{M\rho}{2d}>C_1l,
	\end{equation}
	where $C_1=C_1(\eta_0)$ is the constant appearing in Lemma~\ref{lemma:stimaLinea}.
	
	\item By applying Proposition~\ref{lemma:local_rigidity_alpha}, we obtain  $\bar\eta=\bar{\eta}(M,l)$  and ${\tau}_1 = {\tau}_1(M,l,\delta/2)$.  Thus we fix
	\begin{equation}\label{eq:etafix}
		0<\eta<\bar{\eta},\quad \bar\eta=\bar{\eta}(M,l).
	\end{equation}

	\item Finally, we choose  $\bar\tau>0$ s.t.
	\begin{equation}
		\label{eq:taufix0}
		\begin{split}
			\bar\tau<\min\{\tau_0,\tau_C\}  \qquad\text{as in Lemma~\ref{rmk:stimax1},}
		\end{split}
	\end{equation}
\begin{equation*}
	\bar{\tau}\leq\hat{\tau}\qquad\text{as in Theorem~\ref{thm:convex}},
\end{equation*}
	\begin{equation}
		\label{eq:taufix1}
		\bar\tau<{\tau}_2, \,{\tau}_2\text{ as in Proposition~\ref{lemma:stimaContributoVariazionePiccola} and Lemma~\ref{lemma:stimaLinea}},
	\end{equation}
	\begin{equation}
		\label{eq:taufix2}
		\bar\tau<{\tau}_1, \text{ ${\tau}_1$ as in Proposition~\ref{lemma:local_rigidity_alpha} depending on $M,l,\delta/2$}.
	\end{equation}
	
\end{enumerate}

Notice that $\bar \tau$ depends on $\bar \alpha$ through the dependence of $\tau_1$ on $l$ and the dependence \eqref{eq:lfix} of the intermediate scale $l$ on $\bar \alpha$.   

Let $E$ be a minimizer of $\FtL$ in the class $\mathcal C_{L,\alpha}$. By $[0,L)^d$-periodicity  of $E$ we will denote by $[0,L)^d$  the cube of size $L$ with the usual identification of the boundary.

\subsection{ Decomposition of $[0,L)^d$} \label{subsec:dec}

Now we perform a decomposition of $[0,L)^d$ into different sets according to closeness of the minimizer $E$ to stripes in different directions or deviations from being one-dimensional. The construction of this decomposition is initially analogous to that considered in~\cite{dr_arma,dr_siam,ker}, but in the end it will differ from it due to the different role played in this case by boundary points along slices in directions $e_i$ of sets where the minimizer is close to stripes with boundaries orthogonal to $e_i$ (see definitions~\eqref{eq:ail},~\eqref{eq:albl}).

Let us now consider any $L>l$ of the form $L=2kh^*_{\tau,\alpha}$, with $k\in\N$ and $\alpha$, $\tau\leq\bar{\tau}$ as in Section~\ref{Ss:param}.  We will have that
$[0,L)^d =A_{-1}\cup A_0 \cup (B\setminus B_l)\cup A_{1,l}\cup\ldots \cup A_{d,l}$ where
\begin{itemize}
	\item $A_{i,l}$ with $i > 0$ is made of points $z$ such that there is only one direction $e_i$ such that $E_\tau\cap Q_{l}(z)$ is close to stripes with boundaries orthogonal to~$e_i$. 
	\item $A_{-1}$ is a set of points $z$ such that $E_\tau\cap Q_{l}(z)$ is close both to stripes with boundaries orthogonal to $e_i$ and to stripes with boundaries orthogonal to $e_j$ for some $i\neq j$.  In particular, by Remark~\ref{rmk:lip} (ii) one has that either $|E_\tau\cap Q_l(z)|\ll l^d$ or $|E_{\tau}^c\cap Q_l(z)|\ll l^d$ (see ).
	\item $B\setminus B_l$ is a suitable set of points close to the boundaries of the sets $A_{i,l}$ as $i\in\{1,\dots,d\}$. 
	\item $A_{0}$ is a set of points $z$ where none of the above points is true.
\end{itemize}

The aim is then to show that $A_0\cup A_{-1}\cup B\setminus B_l = \emptyset$ and  that there exists only one $A_{i,l}$ with $i >  0$.

Let us first define the sets $A_i$, for $i\in\{-1,0,1,\ldots,d\}$.

We preliminarily define
\begin{equation*}
	\begin{split}
		\tilde{A}_{0}:= \insieme{ z\in [0,L)^d:\ D_{\eta}(E,Q_{l}(z)) \geq \delta }.
	\end{split}
\end{equation*}
Hence, by the choice of $\delta,M$ made in Section~\ref{Ss:param} and by Proposition~\ref{lemma:local_rigidity_alpha}, for every $z\in \tilde{A}_{0}$ one has that $\bar{F}_{\tau}(E,Q_{l}(z)) > M$.

Let us denote by $\tilde{A}_{-1}$ the set
\begin{equation*}
	\begin{split}
		\tilde{A}_{-1}: = \insieme{z\in [0,L)^d: \exists\, i,j \text{ with } i\neq j \text{ \st }\, D^{i}_{\eta} (E,Q_{l}(z))\leq\delta , D^{j}_{\eta} (E,Q_{l}(z)) \leq \delta }.
	\end{split}
\end{equation*}

Since $\delta$ satisfies~\eqref{eq:deltafix2}, when $z\in \tilde{A}_{-1}$, then one has that $\min (|E\cap Q_{l}(z)|, |Q_{l}(z)\setminus E|) \leq  l^{d-1} $.
Thus, using  Lemma~\ref{lemma:stimaQuasiPieno} with $\delta=1/l$, one has that
\begin{equation*}
	\begin{split}
		\bar{F}_{\tau}(E, Q_{l}(z)) \geq  -\frac{d}{l\eta_0}.
	\end{split}
\end{equation*}

Now we show that the sets $\tilde A_0$ and $\tilde A_{-1}$ can be enlarged while keeping analogous properties.

By the choice of $\rho$ made in~\eqref{eq:rhofix},~\eqref{eq:rhofix2} holds, namely for every $z\in \tilde{A}_{0}$ and $|z- z' |_\infty\leq\rho$ one has that $D_{\eta}(E,Q_{l}(z')) > \delta/2$.

Moreover, let now $z'$ such that $|z- z' |_\infty\leq 1$ with $z\in \tilde{A}_{-1}$. It is not difficult to see that if $|Q_{l}(z)\setminus E | \leq l^{d-1}$ then $|Q_{l}(z')\setminus E| \lesssim l^{d-1}$. Thus from Lemma~\ref{lemma:stimaQuasiPieno}, one has that
\begin{equation}
	\label{eq:tildeC}
	\begin{split}
		\bar{F}_{\tau}(E, Q_{l}(z')) \geq -\frac{\tilde C_d}{l\eta_0}.
	\end{split}
\end{equation}

The above observations motivate the following definitions
\begin{align}
	A_{0} &:= \insieme{ z' \in [0,L)^d: \exists\, z \in \tilde{A}_{0}\text{ with }|z-z'|_{\infty}  \leq \rho }\label{a0}\\
	A_{-1} &:= \insieme{ z' \in [0,L)^d: \exists\, z \in \tilde{A}_{-1}\text{ with }|z-z' |_{\infty}  \leq 1 },\label{a1}
\end{align}

%
%

By the choice of the parameters and the observations above, for every $z\in A_{0}$ one has that $\bar{F}_{\tau}(E,Q_{l}(z)) > M$ and for every $z\in A_{-1}$, $\bar{F}_{\tau}(E,Q_{l}(z)) \geq-\tilde C_d/(l\eta_0)$.

Let us denote by $A:= A_{0}\cup A_{-1}$. 

The set $[0,L)^d\setminus A$ has the following property: for every $z\in [0,L)^d\setminus A$, there exists $i\in \{ 1,\ldots,d\}$ such that $D^{i}_{\eta}(E,Q_{l}(z)) \leq \delta$ and for every $k\neq i$ one has that $D^{k}_{\eta}(E,Q_{l}(z)) > \delta$.

Given  that $A$ is closed, we consider the connected components $\mathcal C_{1},\ldots,\mathcal C_{n}$ of $[0,L)^d\setminus A$.  The sets $\mathcal C_{i}$ are path-wise connected. 
Moreover, given a connected component $\mathcal C_{j}$ one has that there exists  $i$ such that $D^{i}_{\eta}(E,Q_{l}(z)) \leq \delta$ for every $z\in\mathcal  C_{j}$  and for every $k\neq i$ one has that $D^{k}_{\eta}(E,Q_{l}(z)) > \delta$.  
We will say that $\mathcal C_j$ is oriented in direction $e_i$ if there is a point in $z\in \mathcal C_j$ such that $D^{i}_\eta(E,Q_{l}(z)) \leq \delta$. 
Because of the above being oriented along direction $e_{i}$ is well-defined.

We will denote by $A_{i}$ the union of the connected  components $\mathcal C_{j}$ such that $\mathcal C_{j}$ is oriented along the direction $e_{i}$. 

We observe the following

\begin{enumerate}[(a)]
	\item The sets $A=A_{-1}\cup A_{0}$, $A_{1}$, $A_{2}$, $\ldots, A_d$  form a partition of $[0,L)^d$. 
	\item The sets $A_{-1}, A_{0}$ are closed and $A_{i}$, $i>0$, are open.  
	\item For every $z\in A_{i}$, we have that $D^{i}_{\eta}(E,Q_{l}(z)) \leq \delta$. 
	\item  There exists $\rho$ (independent of $L,\tau$) such that  if $z\in A_{0}$, then $\exists\,z'$ s.t. $Q_{\rho}(z')\subset A_{0}$ and $z \in Q_{\rho}(z')$. If $z\in A_{-1}$ then $\exists\,z'$ s.t. $Q_{1}(z')\subset A_{-1}$ and $z \in Q_{1}(z')$. 
	\item For every $z\in {A}_{i}$ and $z'\in {A}_{j}$ one has that there exists a point $\tilde{z}$ in the segment connecting $z$ to $z'$ lying in ${A}_{0}\cup A_{-1}$. 
\end{enumerate}

Let now $B = \bigcup_{i> 0}A_{i}$, $A=A_0\cup A_{-1}$. 

From conditions $(b)$ and $(e)$ above, $B_{t^{\perp}_{i}}$ is a finite union of intervals, each belonging to some $A_{i,t_i^\perp}$, $i\in\{1,\dots,d\}$. Moreover, by $(d)$, for every point that does not belong to $B_{t^\perp_{i}}$ there is a neighbourhood of fixed positive size that is not included in $B_{t^\perp_i}$. 
Let $\{ I^j_{1},\ldots,I^j_{n(j,t_i^\perp)}\}$ such that $\bigcup_{\ell=1}^{n(j,t_i^\perp)} I^j_{\ell} = A_{j,t_i^\perp}$ with $I^j_\ell \cap I^i_{k} = \emptyset$ whenever $j\neq i$ or $j=i$ and $\ell\neq k$. 
We can further assume that $I^j_{\ell} \leq I^j_{\ell+1}$, namely that for every $s\in I^j_{\ell}$ and $s'\in I^j_{\ell+1}$ it holds $s \leq s'$. 
By construction there exists $J_{k} \subset A_{t^{\perp}_{i}}$ such that $I^j_{\ell}\leq  J_{k} \leq I^j_{\ell+1}$, for every $\ell, j$. We set $\bar n(t_i^\perp)=\sum_{j=1}^dn(j,t_i^\perp)$ to be the number of such disjoint intervals $J_k\subset A_{t_i^\perp}$. Whenever $J_k \cap A_{0,t_i^\perp}\neq\emptyset$,  we have that $|J_k | > \rho$  and whenever $J_{k} \cap A_{-1,t^\perp_i}\neq \emptyset $ then $|J_{k}| > 1$. 

Given $i\in\{1,\dots,d\}$, $\ell\in\{1,\dots,n(i,t_i^\perp)\}$ and $I^i_\ell=(a^i_\ell,b^i_\ell)$, define $I^i_{\ell,l}=(a^i_\ell+l/2,b^i_\ell-l/2)$ whenever $|b^i_\ell-a^i_\ell|> l$ and $I^i_{\ell,l}=\emptyset$ otherwise. Analogously set 
\begin{equation}\label{eq:ail}
	A_{i,t_i^\perp,l}=\underset{\ell=1}{\overset{n(i,t_i^\perp)}{\bigcup}} I^i_{\ell,l}
\end{equation}
and 
\begin{equation}\label{eq:albl}
	A_{i,l}=\underset{\{t_{i}^\perp\in[0,L)^{d-1}\}}{\bigcup}A_{i,t_i^\perp,l},\qquad B_{l}=\underset{i=1}{\overset{d}{\bigcup}}A_{i,l}.
\end{equation} 
Set also $n(i,t_i^\perp,l)=n(i,t_i^\perp)-\#\{\ell:\,I^i_{\ell,l}=\emptyset\}$.

Thus we get our final partition $[0,L)^d=A_0\cup A_{-1}\cup (B\setminus B_l)\cup A_{1,l}\cup\ldots\cup A_{d,l}$.

\subsection{Proof of Theorem~\ref{thm:main}}

\textbf{Step 1} First we show that the following estimate holds 

\begin{align}
		\frac{1}{L^d} \int_{B_{t^{\perp}_{i}}} \bar{F}_{i,\tau}(E,&Q_{l}(t^{\perp}_{i}+se_i))\ds + \frac1{dL^d} \int_{A_{t^{\perp}_{i}}}\bar{F}_{\tau}(E,Q_{l}(t^{\perp}_{i}+se_i)) \ds  \notag\\
		&\geq \frac{\Lambda(\tau,\alpha(A_{i,t^\perp_i,l}\times Q_l^\perp(t_i^\perp)))|A_{i,t^{\perp}_{i},l}|}{L^d} - \frac{C_0n(i,t_i^\perp,l)}{L^d}- C(d,\eta_0) \frac{|A_{t^\perp_i}|}{l L^d}.	\label{eq:toBeShown_slice}
\end{align}

By the  definitions given in Section~\ref{subsec:dec}, one has that
\begin{equation*}
	\begin{split}
		\frac{1}{L^d} \int_{B_{t^{\perp}_{i}}} \bar{F}_{i,\tau}(E,&Q_{l}(t^{\perp}_{i}+se_i)) \ds  + 
		\frac{1}{d L^d} \int_{A_{t^{\perp}_{i}}} \bar{F}_{\tau}(E,Q_{l}(t^{\perp}_{i}+se_i)) \ds 
		\\ & \geq \sum_{j=1}^d\sum_{\ell=1}^{n(j,t_i^\perp)}  \frac{1}{L^d}\int_{I^j_{\ell}} \bar{F}_{i,\tau}(E,Q_{l}(t^{\perp}_{i}+se_i)) \ds
		+ \frac{1}{dL^d}\sum_{\ell=1}^{\bar n(t_i^\perp)} \int_{J_{\ell}} \bar{F}_{\tau}(E,Q_{l}(t^{\perp}_i+se_i)) \ds 
		\\ & \geq \frac{1}{L^d}\sum_{j=1}^d\sum_{\ell=1}^{n(j,t_i^\perp)} \Big( \int_{I^j_{\ell}} \bar{F}_{i,\tau}(E,Q_{l}(t^{\perp}_{i}+se_i)) \ds
		+ \frac{1}{2d} \int_{J_{k(j,\ell)-1}\cup J_k(j,\ell)} \bar{F}_{\tau}(E,Q_{l}(t^{\perp}_i+se_i)) \ds\Big),
	\end{split}
\end{equation*}
where in the second inequality we have used the $[0,L)^d$-periodicity and the convention $J_1:=J_{\bar n(t_i^\perp)}$.

Let us first consider $I^i_{\ell} \subset A_{i,t_i^\perp}$.  
By construction, we have that $\partial I^i_{\ell}\subset A_{t^\perp_i}$. 

If $\partial I^i_{\ell}\subset A_{-1,t^\perp_i}$, by using our choice of parameters we can apply~\eqref{eq:gstr27} in Lemma~\ref{lemma:stimaLinea} and obtain
\begin{equation*}
	\begin{split}
		\frac{1}{L^d}\int_{I^i_{\ell}} \bar{F}_{i,\tau}(E,Q_{l}(t^{\perp}_{i}+se_i))\ds  \geq \frac{1}{L^d}\Big(\Lambda(\tau,\alpha(I^i_{\ell,l}e_i+ Q_l^\perp(t_i^\perp)))|I^i_{\ell,l}| - C_0 -\frac{C_1} l\Big).
	\end{split}
\end{equation*}

If $\partial I^i_\ell \cap A_{0, t^\perp_i}\neq \emptyset$, by using our choice of parameters, namely~\eqref{eq:lfix} and~\eqref{eq:taufix1}, we can apply~\eqref{eq:gstr36} in Lemma~\ref{lemma:stimaLinea}, and obtain
\begin{equation*}
	\begin{split}
		\frac{1}{L^d}\int_{I^i_{\ell}} \bar{F}_{i,\tau}(E,Q_{l}(t^{\perp}_{i}+se_i))\ds \geq\frac{1}{L^d}\Big( \Lambda(\tau,\alpha(I^i_{\ell,l}e_i+Q_l^\perp(t_i^\perp)))| I^i_{\ell,l}|-C_1 l\Big).
	\end{split}
\end{equation*}

On the other hand, if $\partial I^i_\ell \cap A_{0,t^\perp_i}\neq \emptyset$, we have that either $J_{k(i,\ell)}\cap A_{0,t^\perp_i}\neq \emptyset$ or $J_{k(i,\ell)-1}\cap A_{0,t^\perp_i}\neq\emptyset$. Thus
\begin{equation*}
	\begin{split}
		\frac{1}{2dL^d}\int_{J_{k(i,\ell)-1}} \bar{F}_{\tau}(E,Q_{l}(t^{\perp}_{i}+se_i)) \ds & + \frac{1}{2dL^d}\int_{J_{k(i,\ell)}} \bar{F}_{\tau}(E,Q_{l}(t^{\perp}_{i}+se_i)) \ds  \\ &\geq  \frac{M\rho}{2dL^d}  - \frac{|J_{k(i,\ell)-1}\cap A_{-1,t^\perp_i} |\tilde C_d}{2dl\eta_0 L^d} - \frac{|J_{k(i,\ell)}\cap A_{-1,t^\perp_i} |\tilde C_d}{2dl \eta_0L^d},
	\end{split}
\end{equation*}
where $\tilde C_d$ is the  constant in~\eqref{eq:tildeC}.

Since $M$ satisfies~\eqref{eq:Mfix}, in both cases $\partial I^i_{\ell}\subset A_{-1,t^\perp_i}$ or $\partial I^i_{\ell}\cap A_{0,t^\perp_i}\neq \emptyset$, we have that
\begin{equation*}
	\begin{split}
		\frac{1}{L^d}\int_{I^i_{\ell}} &\bar{F}_{i,\tau}(E,Q_{l}(t^\perp_i+se_i)) \ds + 
		\frac{1}{2dL^d}\int_{J_{k(i,\ell)-1}}  \bar{F}_{\tau}(E,Q_{l}(t^{\perp}_{i}+se_i))\ds
		+ \frac{1}{2dL^d}\int_{J_{k(i,\ell)}}  \bar{F}_{\tau}(E,Q_{l}(t^{\perp}_{i}+se_i))\ds\\
		&\geq \frac{\Lambda(\tau,\alpha(I^i_{\ell,l}e_i+Q_l^\perp(t_i^\perp))) |I^i_{\ell,l} |}{L^d} -\frac{C_0}{L^d}- \frac{|J_{k(i,\ell)-1}\cap A_{-1,t^\perp_i}|\tilde C_d}{2dl\eta_0L^d}
		- \frac{|J_{k(i,\ell)}\cap A_{-1,t^\perp_i}|\tilde C_d}{2dl\eta_0L^d}.
	\end{split}
\end{equation*}

If $I^j_{\ell} \subset A_{j,t^\perp_i}$ with $j \neq i$ from  Lemma~\ref{lemma:stimaLinea} Point (i) it holds
\begin{equation*}
	\begin{split}
		\frac 1{L^d}\int_{I^j_{\ell}} \bar{F}_{i,\tau}(E,Q_{l}(t^{\perp}_{i}+se_i)) \ds\geq  - \frac{C_1}{lL^d}.
	\end{split}
\end{equation*}

In general for every $J_{k}$  we have that 
\begin{equation*}
	\begin{split}
		\frac{1}{dL^d}\int_{J_{k}} \bar{F}_{\tau}(E,Q_{l}(t^{\perp}_{i}+se_i))\, \ds \geq   \frac{|J_{k}\cap A_{0,t^\perp_i} | M}{dL^d} - \frac{\tilde C_d}{dl\eta_0L^d }|J_{k}\cap A_{-1,t^\perp_i}|. 
	\end{split}
\end{equation*}

For $I^j_{\ell}\subset A_{j,t^\perp_i}$ such that $(J_{k(j,\ell)} \cup J_{k(j,\ell)-1})\cap A_{0,t^\perp_i}\neq \emptyset$ with $j\neq i$, we have that 
\begin{equation*}
	\begin{split}
		\frac{1}{L^d}\int_{I^j_{\ell}} \bar{F}_{i,\tau}(E,Q_{l}(t^\perp_i+se_i)) \ds &+ 
		\frac{1}{2dL^d}\int_{J_{k(j,\ell)-1}}  \bar{F}_{\tau}(E,Q_{l}(t^{\perp}_{i}+se_i))\ds
		+ \frac{1}{2dL^d}\int_{J_{k(j,\ell)}}  \bar{F}_{\tau}(E,Q_{l}(t^{\perp}_{i}+se_i))\ds\\
		&\geq -\frac{C_1}{lL^d} + \frac{M\rho}{2dL^d} - \frac{|J_{k(j,\ell)-1}\cap A_{-1,t^\perp_i}|\tilde C_d}{2dl\eta_0L^d}
		- \frac{|J_{k(j,\ell)}\cap A_{-1,t^\perp_i}|\tilde C_d}{2dl\eta_0L^d}.
		\\ &\geq
		- \frac{|J_{k(j,\ell)-1}\cap A_{-1,t^\perp_i}|\tilde C_d}{2dl\eta_0L^d}
		- \frac{|J_{k(j,\ell)}\cap A_{-1,t^\perp_i}|\tilde C_d}{2dl\eta_0L^d}.
	\end{split}
\end{equation*}
where the last inequality is true due to~\eqref{eq:Mfix}.

For $I^j_{\ell}\subset A_{j,t^\perp_i}$ such that $(J_{k(j,\ell)} \cup J_{k(j,\ell)-1})\subset  A_{-1,t^\perp_i}$ with $j\neq i$, we have that 
\begin{equation*}
	\begin{split}
		\frac{1}{L^d}\int_{I^j_{\ell}} \bar{F}_{i,\tau}(E,Q_{l}(t^\perp_i+se_i)) \ds &+ 
		\frac{1}{2dL^d}\int_{J_{k(j,\ell)-1}}  \bar{F}_{\tau}(E,Q_{l}(t^{\perp}_{i}+se_i))\ds
		+ \frac{1}{2dL^d}\int_{J_{k(j,\ell)}}  \bar{F}_{\tau}(E,Q_{l}(t^{\perp}_{i}+se_i))\ds\\
		&\geq -\frac{C_1}{lL^d}   - \frac{|J_{k(j,\ell)-1}\cap A_{-1,t^\perp_i}|\tilde C_d}{2dl\eta_0L^d}
		- \frac{|J_{k(j,\ell)}\cap A_{-1,t^\perp_i}|\tilde C_d}{2dl\eta_0L^d}.
		\\ &\geq
		- \max\Big(C_1,\frac{\tilde C_d}{\eta_0d}\Big)\bigg(\frac{|J_{k(j,\ell)-1}\cap A_{-1,t^\perp_i}|}{lL^d}
		+ \frac{|J_{k(j,\ell)}\cap A_{-1,t^\perp_i}|}{lL^d}\bigg).
	\end{split}
\end{equation*}
where in the last inequality we have used that $|J_{k(j,\ell)}\cap A_{-1,t^\perp_i}|\geq1, \,|J_{k(j,\ell)-1}\cap A_{-1,t^\perp_i}|\geq1$.

Summing over $j\in\{1,\dots,d\}$, and taking
\begin{equation}\label{eq:ca0}
	C(d,\eta_0)=\max\Big(C_1,\frac{\tilde C_d}{\eta_0d}\Big), 
\end{equation} one obtains~\eqref{eq:toBeShown_slice} as desired.

\textbf{Step 2}

Our aim is to deduce from~\eqref{eq:toBeShown_slice} the following lower bound

\begin{equation}\label{eq:eqtoBeshown_integral}
	\Fcal_{\tau,L}(E)\geq\Lambda(\tau,\alpha(B_l))\frac{|B_l|}{L^d}-\frac{\bar C}{lL^d}|B_l^c|,
\end{equation}
where $\bar C=C_0+dC(d,\eta_0)$.

Integrating~\eqref{eq:toBeShown_slice} w.r.t. $t_i^\perp\in[0,L)^{d-1}$ and using convexity of $\alpha\mapsto\Lambda(\tau,\alpha)$ proved in Theorem \ref{thm:convex} and periodicity w.r.t. $[0,L)^{d-1}$ one has that

\begin{align}
	\frac{1}{L^d}\int_{A_i}\bar F_{i,\tau}(E,Q_l(z))\dz+\frac{1}{dL^d}\int_A\bar F_{\tau}(E,Q_l(z))\dz&\geq\Lambda(\tau,\alpha(A_{i,l}))|A_{i,l}|\notag\\
	&-\frac{C_0}{L^d}\int_{[0,L)^{d-1}}n(i,t_i^\perp,l)\dt_i^\perp\notag\\
	&-\frac{C(d,\eta_0)}{lL^d}|A|.
\end{align}
Summing the above over $i\in\{1,\dots,d\}$, by  the lower bound~\eqref{eq:gstr14} and using the convexity of $\alpha\mapsto\Lambda(\tau,\alpha)$ and the definition of the sets in the decomposition  one obtains
\begin{align}
\Fcal_{\tau,L}(E)\geq\Lambda(\tau,\alpha(B_l))\frac{|B_l|}{L^d}-\frac{C_0}{L^d}\sum_{i=1}^d\int_{[0,L)^{d-1}}n(i,t_i^\perp,l)\dt_i^\perp
-\frac{dC(d,\eta_0)}{lL^d}|A|.
\end{align}
Observing that $|B\setminus B_l|=l\sum_{i=1}^d\int_{[0,L)^{d-1}}n(i,t_i^\perp,l)\dt_i^\perp$, one gets~\eqref{eq:eqtoBeshown_integral}.

\textbf{Step 3}

Now let $L=2kh^*_{\tau,\alpha}$, satisfying   $\Lambda(\tau,\alpha)=\Fcal_{\tau,L}(E)$ and assume that   $\frac{|B_l^c|}{L^d}\neq0$. 



One has that, by convexity of the map $\alpha\mapsto\Lambda(\tau,\alpha)$ and the fact that 
\[
(\alpha(B_l)-\alpha)\frac{|B_l|}{L^d}=(\alpha-\alpha(B^c_l))\frac{|B^c_l|}{L^d},
\] 
it holds
\begin{align}
	\Lambda(\tau,\alpha)&\geq \Lambda(\tau,\alpha(B_l))\frac{|B_l|}{L^d}-\frac{\bar C|B_l^c|}{lL^d}\notag\\
	&\geq \Lambda(\tau,\alpha)\frac{|B_l|}{L^d}+\partial_\alpha\Lambda(\tau,\alpha)(\alpha-\alpha(B^c_l))\frac{|B^c_l|}{L^d}-\frac{\bar C|B_l^c|}{lL^d}.\label{eq:ineqminore}
\end{align}
Hence, dividing both terms of~\eqref{eq:ineqminore} by $\frac{|B^c_l|}{L^d}\neq0$ and using the fact that $\partial_\alpha\Lambda(\tau,\alpha)\leq0$,
\begin{align}
		\Lambda(\tau,\alpha)-\alpha\partial_\alpha\Lambda(\tau,\alpha)\geq-\partial_\alpha\Lambda(\tau,\alpha)\alpha(B^c_l)-\frac{\bar C}{l}\geq-\frac{\bar C}{l}.
\end{align}
By the Taylor formula for $\alpha\mapsto\Lambda(\tau,\alpha)$ between $0$ and $\alpha$, it follows that
\begin{align}
	\Lambda(\tau,0)-\frac{1}{2}\int_0^\alpha\partial^2_z\Lambda(\tau,z)z^2\dz\geq-\frac{\bar C}{l}.\label{eq:lastineq}
\end{align}
Now observe that $\Lambda(\tau,0)=0$, $\alpha\geq\bar \alpha$ and that by \eqref{eq:2derineq} on $[\bar{\alpha}/2,\bar\alpha]$ it holds $\partial_z^2\Lambda(\tau,z)\geq  c{\bar{\alpha}}^{q-1}$.

Thus provided 
\begin{equation}\label{eq:lfinal} 
	l\geq \frac{c_q\bar C}{\bar{\alpha}^{q+2}}
	\end{equation} as in~\eqref{eq:lfix} for a suitable constant $c_q$ depending only on $q$ one gets a contradiction in~\eqref{eq:lastineq}.

Hence  one has that $\frac{|B_l^c|}{L^d}=0$. In particular, $|A|\leq|B^c_l|=0$ and thus by $(v)$ there is just one $A_i$, $i>0$ with $|A_i|>0$.

We claim that this proves the statement of Theorem~\ref{thm:main}.

Indeed, let us consider 
\begin{align}
	\frac{1}{L^d}\int_{[0,L)^d}\bar F_{\tau}(E,Q_l(z))\dz&=\frac{1}{L^d}\int_{[0,L)^d}\bar F_{i,\tau}(E,Q_l(z))\dz\label{eq:fi}\\
	&+\frac{1}{L^d}\sum_{j\neq i}\int_{[0,L)^d}\bar F_{j,\tau}(E,Q_l(z))\dz\label{eq:fj}
\end{align}

We apply now Lemma~\ref{lemma:stimaLinea} with $j =i$ and slice the cube $[0,L)^d$ in direction $e_i$. 
From~\eqref{eq:gstr21}, one has that~\eqref{eq:fj} is nonnegative and strictly positive unless the set $E$ is a union of stripes with boundaries orthogonal to $e_i$.
On the other hand, from~\eqref{eq:gstr28}, one has the \rhs of~\eqref{eq:fi} is minimized by a periodic union of stripes with boundaries orthogonal to $e_i$ and with period  $2h^*_{\tau,\alpha}$ and density $\alpha$. Thus, periodic stripes of period $2h^*_{\tau,\alpha}$ and density $\alpha$ are optimal.

\printbibliography

\end{document}